\documentclass[12pt]{article}
\usepackage[utf8]{inputenc}
\usepackage{amsfonts}
\DeclareMathSymbol{\shortminus}{\mathbin}{AMSa}{"39}
\usepackage{ amssymb }
\usepackage{mathpazo}
\usepackage{ stmaryrd }
\usepackage{amsthm}
\usepackage{mathtools}
\usepackage{amsmath}

\newtheorem{thm}{Theorem}

\newtheorem{cor}{Corollary}

\newtheorem{lem}{Lemma}

\newtheorem{mydef}{Definition}

\newtheorem{ex}{Example}

\newtheorem{remark}{Remark}

\newtheorem{prop}{Proposition}

\title{A solution to the $p$-adic Schr\"{o}dinger equation}
\author{Haonan Gu}
\date{December 2022}

\begin{document}

\maketitle

\section{Introduction}

The field of $p$-adic numbers $\mathbb{Q}_p$ (with $p$ a prime number) was first introduced by the mathematician Kurt Hensel at the end of the 19th century. It provides an alternative number system to the field of real numbers $\mathbb{R}$ that singles out a prime $p$. Because of its unique properties as well as topological and algebraic structure, the field of $p$-adic numbers has wide applications in number theory, with Andrew Wiles' proof of Fermat's Last Theorem being among the most famous ones. This thesis aims to provide an introduction to the basic properties of the field of $p$-adic numbers and explore a $p$-adic analogue to the famous time-independent Schr\"{o}dinger equation. 

In the second section of this thesis, we start by defining the $p$-adic norm $|\cdot|_p$ on $\mathbb{Q}_p$, an alternative way of measuring distance compared to the real absolute value. Under this view of distance, two rational numbers are closer to one another if their difference is divisible by a higher power of $p$. After establishing some topological definitions and proofs following \cite{Alexa} \cite{Jack} \cite{Zuniga}, we define the $p$-adic field $\mathbb{Q}_p$ as the completion of $\mathbb{Q}$ with respect to the $p$-adic norm $|\cdot|_p$. We prove the existence and uniqueness of the $p$-adic expansion for each $p$-adic number, which is analogous to the decimal expansion of a real number in the sense that every $p$-adic number can be expressed as a power series in $p$. Finally, basic topological properties of the $p$-adic field are proved and the ring of $p$-adic integers $\mathbb{Z}_p$ is defined in Section 2.3. Using the local compactness of $\mathbb{Q}_p$, we define a notion of $p$-adic integration in Section 3 by considering a suitable Haar measure. We prove some important propositions on $p$-adic integration, as well as compute specific examples of $p$-adic integrals.

In Section 4, we introduce a $p$-adic derivative operator $D^{\alpha}f(x)$ in \cite{Vladimirov}. Similar to the derivative in $\mathbb{R}$, it satisfies 
\begin{equation*}
D^\alpha [D^\beta f(x)]=D^{\alpha+\beta} f(x)
\end{equation*}
for any $\alpha,\beta\in\mathbb{N}$ and complex-valued functions $f$ of a $p$-adic variable. Another important property of operator $D^{\alpha}$ is 
\begin{equation*}
D^\alpha {\left| x\right|_p}^n=\frac{\Gamma_p(n+1)}{\Gamma_p(n-\alpha+1)} \left| x\right|_p^{n-\alpha} 
\end{equation*}
for n$\geq$1. 

Using the notion of $p$-adic integration introduced in Section 3 and the derivative operator in Section 4, we finally explore our main application in Section 5. In 5.1, we introduce the time-independent Schr\"{o}dinger equation
\begin{equation*}
-\frac{h^2}{2m}\frac{d^2\Psi}{dx^2}+\frac{1}{2}m\omega^2x^2\Psi=E\Psi,
\end{equation*} its ground state solution \begin{equation*}
\Psi_0(x)=A_0e^{-\frac{m\omega}{2h}x^2},
\end{equation*} where $A_0$ is some constant, and the raising and lowering operators $a_+$ and $a_-$, which produce solutions to the equation with raising or lowering energy $E$. In light of the traditional Schr\"{o}dinger equation, we define a $p$-adic analogue
\begin{equation*}
D^2\Psi(x)+B\left| x\right|_p^2\Psi(x)=E\Psi(x)
\end{equation*}
using the derivative operator $D^{2}$ introduced in Section 4. Here B and E are some complex constants. We make two attempts to find a solution to this equation with the first one more naive, and the second one successful. Inspired by \begin{equation*}
\Psi_0(x)=A_0e^{-\frac{m\omega}{2h}x^2},
\end{equation*} the ground solution of the traditional Schr\"{o}dinger equation, we make a naive guess that a solution of the $p$-adic Schr\"{o}dinger equation is in the form
$$\Psi_0(x)=\sum_{n=0}^\infty b_{2n}\left| x\right|_p^{2n}$$
and solve for the coefficients $b_{2n}$. However, a computation reveals that this series won't converge everywhere in $\mathbb{Q}_p$. We then make a second, successful guess by writing a series that has a better chance of convergence
\begin{align*}
\psi_0(x) &= \sum_{n=0}^{\infty} c_n f_n (x) + \sum_{n=1}^{\infty} k_n g_{-n}(x) \\
&= \begin{cases} \sum_{n=0}^{\infty} c_n |x|_p^n & \text{ if } x \in \mathbb{Z}_p \\ \sum_{n=1}^{\infty} k_n |x|_p^{-n} & \text{ otherwise}. \end{cases}
\end{align*} 
After determining the recursive relationships of the series' coefficients, we use asymptotic methods to approximate the eigenvalue E of the solution and the explicit expression of $c_n$ and $k_n$, in the limit of $p$ large. We then successfully prove our series converges everywhere in $\mathbb{Q}_p$ in the case of $B=1$. Using similar methods, we find a solution for the case $B=-1$. 

To summarize, basic definitions and topological properties of the field of $p$-adic numbers are introduced and solutions to the $p$-adic analog of the Schr\"{o}dinger equation are explored in light of the $p$-adic derivative operator and $p$-adic calculus. Whether there exists a $p$-adic version of the raising and lowering operators remains an open question and could be a direction for future research. The power series expression of the solution as this thesis finds could be a good inspiration for finding more solutions.

\section{The field of $p$-adic numbers}

\subsection{The $p$-adic norm}
We start by introducing the definition of a norm on a general field $K$. We then introduce the $p$-adic norm $|\cdot|_p$ on the field of rational numbers $\mathbb{Q}$.

\vspace{0.5cm}

\begin{mydef}
We say $|\cdot|$ is a norm on a field K if for all $x,y\in K$ it satisfies the following properties: 

1. $\left| x\right| \geq0$

2. $\left| x\right|=0$ if and only if $x=0$

3. $\left| xy\right|=\left| x\right|\left| y\right|$

4. $\left| x+y\right|\leq\left| x\right|+\left|y\right|$
\end{mydef}

\vspace{0.5cm}

\begin{thm}
Let $p$ be a prime number, and $x\in\mathbb{Q}\setminus \{0\}$. Then there exists a unique $\gamma\in\mathbb{Z}$ such that x=p$^\gamma$$\frac{m}{n}$, where m,n$\in\mathbb{Z}$ are not divisible by $p$. 
\end{thm}

\begin{proof}
Fix $x\in\mathbb{Q}\setminus \{0 \}$ and a prime number $p$. Then, $x=\frac{r}{q}$ for some co-prime $r,q\in\mathbb{Z}$. 

If $x=1$, then $\gamma=0$ and $m=n=1$ obviously satisfy the condition.
We need to verify $\gamma=0$ is unique. We prove by contradiction by assuming $\gamma$$\neq$0.

Suppose $\gamma>0$, and $n=p^\gamma m$. Then $n$ is divisible by $p$. This contradicts the assumption that $n$ is not divisible by $p$. Suppose instead $\gamma<0$, and $m=p^{\shortminus\gamma}n$. Then $m$ is divisible by $p$. This contradicts the assumption that $m$ is not divisible by $p$.
Thus, for $x=1$, $\gamma=0$ exists and is unique. 

\vspace{0.3cm}
If $q=1$ and $r\neq1$, we must have $n=1$, $x=p^\gamma m$. We use the Fundamental Theorem of Arithmetic and have $\left| x\right|=p_1^{a_1} \dotsb p_k^{a_k}$ for unique prime numbers 
$p_1,\dotsb, p_k$ and unique positive integers $a_1, \dotsb, a_n$.
If $x$ is co-prime to $p$, then $\gamma$ can only equal to $0$.
If $x$ is not co-prime to $p$, or equivalently, $p=p_j$ for some $j \in \{1,\dotsb,n\}$, then $\gamma$ can only equal to $a_j$. If not, $\gamma>a_j$ contradicts with the Fundamental Theorem of Arithmetic and $\gamma<a_j$ leads to $m$ divisible by $p$, contradicting with our assumption.  In this case, we find that $\gamma$ exists and is unique.

\vspace{0.3cm}

If $q\neq1$ and $r=1$, then we apply the Fundamental Theorem of Arithmetic on $q$, we have that $$\left| x\right|=\frac{1}{q}=p_1^{a_1}\dotsb p_k^{a_k}$$
for unique prime numbers $p_1, \dotsb, p_k$ and negative integers $a_1, \dotsb, a_k$. If $q$ is co-prime to $p$, then $\gamma=0$. If $p=p_j$ for some $j$ in $\{1,\dotsb,k\}$, then $\gamma=a_j$. We find that $\gamma$ exists and is unique.

\vspace{0.3cm}

If $q\neq1$, $r\neq1$, we have 
\begin{equation*}
\left| r\right|=p_1^{a_1}\dotsb p_n^{a_n} \text{ and } \left| q\right|=p_{n+1}^{a_{n+1}}\dotsb p_{n+m}^{a_{n+m}}
\end{equation*}

for unique primes $p_1, \dotsb, p_{n+m}$ and unique positive integers 

$a_1, \dotsb, a_{n+m}$. 

Thus, 
\begin{equation*}
\left| x\right|=p_1^{a_1}\dotsb p_k^{a_k}p_{k+1}^{-a_{k+1}}\dotsb p_{k+l}^{-a_{k+l}}.
\end{equation*}
If $p\neq p_j$ for any $j$ in $\{1,\cdots,k\}$, then $\gamma=0$. If $p=p_j$ for some $j$ in $\{1,\dotsb,k\}$, then $\gamma=a_j$. If $j$ is in $\{k+1,\dotsb,k+l\}$, then $\gamma={-a_j}$. Thus, $\gamma$ exists. The uniqueness can be proved by the similar reasoning as before, i.e. any other $\gamma$ will lead to $m$, $n$ not coprime.
\end{proof}

\vspace{0.3cm}

Theorem 1 guarantees that we can define a map on the field $\mathbb{Q}$ as follows.
\begin{mydef}
Define the map $|\cdot|_p: \mathbb{Q}\rightarrow\mathbb{Q}$ as 
\begin{equation*}
|x|_p = \begin{cases} 0 &\text{ if } x=0  \\ p^{\shortminus\gamma} & \text{ otherwise} \end{cases},
\end{equation*}
where $\gamma$ is a constant integer determined by Theorem 1.
\end{mydef}

\begin{ex}

\begin{itemize}
\item[(i)] $\left| 114514\right|_2$=$\left| 2\times57257\right|_2$=2$^{\shortminus1}$
\item[(ii)] $\left| \frac{1919}{810}\right|_5$=$\left| \frac{1919}{5\times162}\right|_5$=5$^1$.
\end{itemize}
\end{ex}

We want to verify that $\left| \cdot \right|_p$ is a norm on $\mathbb{Q}$. Since $\left| \cdot \right|_p$ obviously satisfies conditions 1 and 2 in Definition 1, we check the third and fourth conditions of Definition 1 below.

\begin{thm}
Let $p$ be a prime. The map $|\cdot|_p$ is a norm on $\mathbb{Q}$ (called the $p$-adic norm). Moreover, we have that 
\begin{equation*}
|x+y|_p \leq \text{max}(|x|_p, |y|_p), 
\end{equation*}
for all $x,y \in \mathbb{Q}$, and if $|x|_p \neq |y|_p$ we have 
\begin{equation*}
|x+y|_p = \text{max}(|x|_p, |y|_p).
\end{equation*}
\end{thm}

\begin{proof}  When either $x=0$ or $y=0$, it is easy to observe that $$\left| xy\right|_p=\left| x\right|_p\left| y\right|_p=0.$$

Thus, fix $x, y\in\mathbb{Q}\setminus \left\{ 0 \right\}$.
Since $x=p^{\gamma_1}\frac{m}{n}$ and $y=p^{\gamma_2}\frac{a}{b}$, for some $m,n,a,b\in\mathbb{Z}$ not divisible by $p$, and unique $\gamma_1$, $\gamma_2\in\mathbb{Z}$ we have $$\left| x\right|_p\left| y\right|_p=p^{\shortminus(\gamma_1+\gamma_2)}.$$

Since $xy=p^{(\gamma_1+\gamma_2)}\frac{ma}{nb}$, and $ma, nb$ are not divisible by $p$, we have $$\left| xy\right|_p=p^{\shortminus(\gamma_1+\gamma_2)}.$$
Thus, we have shown $\left| xy\right|_p=\left| x\right|_p\left| y\right|_p$.

\vspace{0.5cm}

We now check the fourth condition. When either $x=0$ or $y=0$, it is easy to see $\left| x+y\right|_p=\left| x\right|_p+\left| y\right|_p$.

\vspace{0.3cm}

Fix $x,y\in\mathbb{Q}\setminus\left\{ 0 \right\}$. We write $x=p^{\gamma_1}\frac{m}{n}$ and $y=p^{\gamma_2}\frac{a}{b}$ for some integers $m,n,a,b$ not divisible by $p$, and unique $\gamma_1$, $\gamma_2\in\mathbb{Z}$. 

When $\gamma_1\neq\gamma_2$, we assume $\gamma_1<\gamma_2$ without loss of generality.
Then,
\begin{equation}
x+y=p^{\gamma_1}\frac{m}{n}+p^{\gamma_2}\frac{a}{b}=p^{\gamma_1}(\frac{m}{n}+p^{\gamma_2-\gamma_1}\frac{a}{b})=p^{\gamma_1}\frac{mb+p^{\gamma_2-\gamma_1}an}{nb}.
\end{equation}
Certainly, $mb+anp^{\gamma_2-\gamma_1}$ and $nb$ are not divisible by $p$. Thus, $$\left| x+y\right|_p=p^{\shortminus\gamma_1}.$$
Since $\left| x\right|_p+\left| y\right|_p=p^{\shortminus\gamma_1}+p^{\shortminus\gamma_2}$ and $p^{\shortminus\gamma_1}$, $p^{\shortminus\gamma_2}>0$, we have $$\left| x+y\right|_p<\left| x\right|_p+\left| y\right|_p.$$
In fact, we see $\left| x+y\right|_p=max(\left| x\right|_p,\left| y\right|_p)$ whenever $\left| x\right|_p\neq\left| y\right|_p$.

When $\gamma_1=\gamma_2$, 
\begin{equation}
x+y=p^{\gamma_1}\frac{m}{n}+p^{\gamma_1}\frac{a}{b}=p^{\gamma_1}\frac{mb+an}{nb}.
\end{equation}
 Since $nb$ is not divisible by $p$, and $mb+na$ is possibly divisible by $p$, then $\gamma$ for $(x+y)$ is greater than or equal to $\gamma_1$.
Thus, 
\begin{equation}
\left| x+y\right|_p=p^{\shortminus\gamma}\leq{p^{\shortminus\gamma_1}}=\left| x\right|_p<{\left| x\right|_p+\left| y\right|_p}.
\end{equation}

This verifies that the function $\left| \cdot \right|_p$ on $\mathbb{Q}$ is a norm. 
\end{proof}

\begin{mydef}
Two norms $|\cdot|_a$, $|\cdot|_b$ on a field $\mathbb{F}$ are equivalent if there exists $\epsilon>0$ such that $|\cdot|_a^\epsilon=|\cdot|_b$.
\end{mydef}

\begin{thm} (Ostrowski's Theorem)
Every nonzero absolute value on $\mathbb{Q}$ is equivalent to either the standard absolute value or one of the $p$-adic norms.
\end{thm}

\begin{remark}
The proof of Theorem 3 can be found on page 3 of \cite{Vladimirov}. This theorem tells us the significance of the $p$-adic norm in understanding the topology of the field of rational numbers. 
\end{remark}

\subsection{The field of $p$-adic numbers}

In this section, we introduce a notion of convergence with respect to the $p$-adic norm, Cauchy sequences, and the completion of $\mathbb{Q}$ with respect to $|\cdot|_p$. We then study properties of the field of $p$-adic numbers $\mathbb{Q}_p$, for example, the $p$-adic expansion of an element in $\mathbb{Q}_p$.

\begin{mydef}
Fix a normed field $(K, \left| \cdot \right|)$. A sequence $(x_n)$ in $K$ is Cauchy with respect to the given norm if for every $\epsilon>0$, there exists some $N\in\mathbb{N}$ such that for all $m,n\geq N$, $$\left| x_n-x_m\right|<\epsilon.$$
\end{mydef}

\begin{mydef}
Fix a normed field $(K, \left| \cdot \right|)$. A sequence $(x_n)$ in $K$ is convergent with respect to the given norm if there exists $b\in K$ such that for every $\epsilon>0$, there exists some $N\in\mathbb{N}$ with $$|x_n-b|<\epsilon$$
for all $n\geq N$.
We call $b$ the limit of the sequence.
\end{mydef}

\begin{mydef}
Fix a normed field $(K, \left| \cdot \right|)$. We say $K$ is complete if every Cauchy sequence has a limit in $K$.
\end{mydef}

\begin{thm}
Fix a normed field $(K, \left| \cdot \right|)$. Then there exists a unique complete normed field $(K', \left| \cdot \right|')$ up to isomorphism that extends $K$. Here $\left| \cdot \right|$' restricts to $\left| \cdot\right|$. Moreover, $K$ is dense in $K'$.
\end{thm}

\begin{proof}
See Theorem 3.4 on p5 of \cite{Alexa}
\end{proof}

By the uniqueness of the completion of a normed field in Theorem 4, we have the following definition.

\begin{mydef}
We denote by $\mathbb{Q}_p$ the completion of $\mathbb{Q}$ with respect to the $p$-adic norm $\left| \cdot \right|_p$. We call $\mathbb{Q}_p$ the field of $p$-adic numbers.
\end{mydef}

\begin{thm}[Convergence implies Cauchy] Any convergent sequence in $\mathbb{Q}_p$ is Cauchy.
\end{thm}

\begin{proof}
Fix $(a_k)_{k\in\mathbb{N}}$ a convergent sequence in $\mathbb{Q}_p$. Then, there exists $a\in\mathbb{Q}_p$ such that for every $\epsilon>0$, there exists $N\in\mathbb{N}$ such that for any $k\geq N$, $$\left| a_k-a\right|_p<\epsilon.$$ 
Fix $\epsilon>0$ so that the corresponding $N$ is fixed. Fix $n,m \geq N$. Then, $\left| a_n-a\right|_p<\epsilon$ and $\left| a_m-a\right|_p<\epsilon$.

We find $$\left| a_n-a_m\right|_p=\left| a_n-a+a-a_m\right|_p\leq max(\left| a_n-a\right|_p, \left| a_m-a\right|_p)<\epsilon.$$
\end{proof}

\begin{thm}
Let $(a_k)_{k\in\mathbb{N}}$ be a sequence in $\mathbb{Q}_p$. Then $(a_k)$ is Cauchy if and only if for every $\epsilon>0$, there exists some $N\in\mathbb{N}$ such that for all $k\geq N$, $$\left| a_{k+1}-a_{k}\right|_p<\epsilon.$$ 
\end{thm}

\begin{proof}
The proof can be found on page 6 of \cite{Alexa}.
\end{proof}

\begin{prop}
For any sequence $(a_k)_{k\in\mathbb{N}\cup \{0\}}$ in $\mathbb{Q}_p$, $\sum_{k=0}^\infty a_k$ converges if and only if $\lim_{k\shortrightarrow\infty}a_k=0$.

\begin{proof}
Fix $(a_k)_{k\in\mathbb{N}}$ a sequence in $\mathbb{Q}_p$,  and $(s_k)$ the sequence of partial sums, i.e. $s_k=a_0 + \dotsb + a_{k-1}$.

($\Rightarrow$) Assume $\sum_{k=0}^\infty a_k$ converges. Then by Theorem 5, the sequence of the partial sums $(s_k)$ is Cauchy. Fix $\epsilon>0$. By Theorem 6, there exists some $N\in\mathbb{N}$ such that for all $k\geq N$, $|s_{k+2}-s_{k+1}|_p=\left| a_{k+1} \right|_p<\epsilon$. Thus, for all $k\geq N+1$, we have $\left| a_k \right|_p<\epsilon$, which is equivalent to $lim_{k\shortrightarrow\infty}a_k=0$.

($\Leftarrow$) Assume $\lim_{k\shortrightarrow\infty}a_k=0$. Fix $\epsilon>0$. Then there exists some $N\in\mathbb{N}$ such that for all $k\geq N$, $\left| a_k \right|_p<\epsilon$. Thus, $\left| s_{k+1}-s_k \right|_p<\epsilon$ for all $k\geq N$. By Theorem 6, $(s_k)$ is Cauchy and therefore converges by the completeness of $\mathbb{Q}_p$. Thus, $\sum_{k=0}^\infty a_k$ converges.
\end{proof}

\end{prop}

\begin{prop}
Fix $p$ as a prime number. Any series in the form $$\sum_{k=0}^\infty x_k p^{k+\gamma},$$ with $\gamma\in\mathbb{Z}$, $x_k \in \mathbb{Z}$ such that $x_0>0$ and $0\leq{x_k}\leq{p-1}$, and $k\in\mathbb{N}\cup \{0\}$, converges in $\mathbb{Q}_p$.
\end{prop}

\begin{proof}
We claim that $lim_{k\shortrightarrow\infty}x_kp^k=0$ with respect to the $p$-adic norm, where $0\leq{x_k}\leq{p-1}$ for any $k\in\mathbb{N}\cup \{0\}$. 

Let $\epsilon>0$. There exists $N\in\mathbb{N}$ such that $\frac{1}{N}<\epsilon$. Fix $k\geq N$. We have $$\frac{1}{p^k}\leq\frac{1}{p^N}<\frac{1}{N}<\epsilon.$$

Thus, $\left| x_kp^k\right|_p=p^{\shortminus k}<\epsilon$. Consequently, $lim_{k\rightarrow\infty}\ x_kp^k=0$ with respect to the $p$-adic norm. 

Fix a series of the form $\sum_{k=0}^\infty x_kp^{k+\gamma}$ with $\gamma\in\mathbb{Z}$, $x_k$ integers such that $0\leq{x_k}\leq{p-1}$, $x_0>0$, and $k\in\mathbb{N}\cup\{0\}$.
By Proposition 1 and the fact that $lim_{k\shortrightarrow\infty}x_kp^{k+\gamma}=0$, we find that $\sum_{k=0}^\infty x_kp^{k+\gamma}$ converges in $\mathbb{Q}_p$.

\end{proof}

\begin{remark}
Equipped with Proposition 1, we proved that a family of power series always converges in $\mathbb{Q}_p$ for some prime number $p$ (Proposition 2 above). It means this expansion gives $p$-adic numbers, but the converse is not necessarily true. This is why we need Proposition 3 to define the $p$-adic expansion for any $x\in\mathbb{Q}_p$.
\end{remark}

\begin{prop}
Every $p$-adic number has a unique $p$-adic expansion.
\end{prop}

\begin{proof}

Fix $x \in\mathbb{Q}_p\setminus \{0\}$. We multiply $x$ by $p^m$ for some $m\in\mathbb{Z}$ such that $$|x_1|_p=|xp^m|_p\leq1.$$

We claim that there exists integers $0\leq a_k\leq p-1$ for $k=0, \dotsb, N$ such that $$|x_1-\Sigma_{k=0}^N a_kp^k|< p^{-N}.$$ Since $p^{-N}$ converges to $0$, proving this claim is sufficient to prove the existence of the $p$-adic expansion by Squeeze Theorem.

We prove the base case that there exists an integer $0\leq a_0 \leq p-1$ such that $|x_1-a_0|_p<1$.

Suppose $|x_1|_p<1$. Then, $|x_1-a_0|_p=max(|x_1|_p, |a_0|_p)$ since $|x_1|_p\neq|a_0|_p$.
Obviously the only $a_0$ that satisfies the condition is $0$. If not, $|a_0|_p=1$, $|x_1-a_0|_p=1$, which is not less than $1$. 

Now suppose $|x_1|_p=1$. We pick a rational number $q$ that satisfies $|x_1-q|_p<1$, and write $q=\frac{d}{e}$, where $d,e$ are co-prime integers.

Since $|x_1-q|_p<1$, $|q|_p=1$. Thus, both $d,e$ are co-prime to $p$. By Bezout's identity, there exist integers $s,t$ such that $se+tp=1$.

Thus, $|x_1-sd|_p=|x_1-c+tpc|_p\leq max(|x_1-c|_p, |tpc|_p)<1$.
Let $l$ be the unique integer that $sd=lp+r$, where $r\in \{0,\dotsb, p-1\}$. Then $$|x_1-r|_p=|x_1-sd+lp|_r\leq max(|x_1-sd|_p, |lp|_p)<1.$$ Thus, $r$ is the $a_0$ that satisfies the base case.

We assume $|x_1-\Sigma_{k=0}^N a_kp^k|_p<p^{-N}$, and need to prove there exists an integer $0\leq a_{N+1}\leq p-1$ such that $$|x_1-\Sigma_{k=0}^{N+1} a_kp^k|_p<p^{-(N+1)}.$$

Given 
\begin{equation*}
|x_1-\Sigma_{k=0}^{N+1}a_kp^k|_p=|x_1-\Sigma_{k=0}^N a_kp^k-a_{N+1}p^{N+1}|_p,
\end{equation*}
we multiply the inequality by $p^{N+1}$, and have 
\begin{equation*}
\Big|(x_1-\sum_{k=0}^N a_kp^k)p^{-N-1}-a_{N+1}\Big|_p<1.
\end{equation*}
The existence of $a_{N+1}$ follows by the base case since 
\begin{equation*}
\Big|(x_1-\sum_{k=0}^N a_kp^k)p^{-N-1}\Big|_p\leq1.
\end{equation*}
Thus, by induction, we conclude that every $p$-adic number has a $p$-adic expansion.

\vspace{0.3cm}

We now claim that such a $p$-adic expansion is unique. Fix a $p$-adic number $q$, and suppose there exists two $p$-adic expansions of $q$, i.e. 
\begin{equation*}
q=\sum_{k=0}^{\infty}x_kp^{k+\gamma_1}=\sum_{k=0}^{\infty}y_kp^{k+\gamma_2}
\end{equation*}
where $x_k,y_k\in\mathbb{Z},$ and $0\leq x_k, y_k\leq n-1$ for all $k\in\mathbb{N}$, and $x_0,y_0>0$, $\gamma_1, \gamma_2\in\mathbb{Z}$.

Without loss of generality, let $\gamma_1\geq\gamma_2$. Then, 
\begin{equation*}
\sum_{k=0}^{\infty} x_kp^{k+\gamma_1}-y_kp^{k+\gamma_2}=-\sum_{k=0}^{\gamma_1-\gamma_2}y_kp^{k+\gamma_2}+\sum_{k=\gamma_1-\gamma_2+1}^{\infty} (x_k-y_k)p^{k+\gamma_2}.    
\end{equation*}

Let $\epsilon>0$. By the definition of convergence, there exists $N\in\mathbb{N}$ such that whenever $n\geq N$, we have 
\begin{equation*}
\Big|-\sum_{k=0}^{\gamma_1-\gamma_2}y_kp^{k+\gamma_2}+\sum_{k=\gamma_1-\gamma_2+1}^{n} (x_k-y_k)p^{k+\gamma_2}\Big|_p<\epsilon.
\end{equation*}

Fix $k'\in\mathbb{N}$. We assume by contradiction that some $y_k$ with $k$ between $0$ to $\gamma_1-\gamma_2$, or some $(x_k-y_k)$, with $k$ between $(\gamma_1-\gamma_2+1)$ and $k'$ is nonzero.

Then, 
\begin{align*}
&\Big|-\sum_{k=0}^{\gamma_1-\gamma_2}y_kp^{k+\gamma_2}+\sum_{k=\gamma_1-\gamma_2+1}^{k'} (x_k-y_k)p^{k+\gamma_2}\Big|_p \\&\quad =max(|y_0p^{\gamma_2}|_p,..., |(x_{k'}-y_{k'})|_p)=p^{-k''} 
\end{align*}
for some integer $k''$ between $0$ and $k'$. The desired equality follows since the $p$-adic norm of any nonzero term is different. 

Let $\epsilon=\frac{p^{-k''}}{2}<p^{-k''}$. 
Then, 
\begin{equation*}
\Big|-\sum_{k=0}^{\gamma_1-\gamma_2}y_kp^{k+\gamma_2}+\sum_{k=\gamma_1-\gamma_2+1}^{N} (x_k-y_k)p^{k+\gamma_2}\Big|_p\geq p^{-k''}>\epsilon 
\end{equation*}
for any $N\geq k'$. 

This contradicts the convergence of the whole series to $0$ in the $p$-adic sense. Thus, all coefficients of the series are $0$. It directly follows that $\gamma_1=\gamma_2$, and $x_k=y_k$ for any $k\in\mathbb{N}$. Hence the $p$-adic expansion of any $p$-adic number exists and is unique.
\end{proof}

\begin{mydef}
Fix $p$ as a prime number. The $p$-adic expansion of a number $x$ in $\mathbb{Q}_p\setminus\{0\}$ is a series of the form $$x=\sum_{k=0}^\infty x_kp^{k+\gamma}$$ with $\gamma\in\mathbb{Z}$, $x_k$ integers such that $0\leq{x_k}\leq{p-1}$, $x_0>0$, and $k\in\mathbb{N}\cup\{0\}$.
\end{mydef}

\begin{prop}
Let $x$ be a $p$-adic number with $p$-adic expansion \begin{equation*}
x=\sum_{k=0}^\infty x_kp^{k+\gamma}
\end{equation*}
 with $\gamma\in\mathbb{Z}$, $x_k$ are integers such that $0\leq{x_k}\leq{p-1}$, $x_0>0$, $k\in\mathbb{N}\cup\{0\}$. Then, $|x|_p=p^{-\gamma}$.
\end{prop}

\begin{proof}
Fix $N\in\mathbb{N}$.
\begin{equation*}
|\sum_{k=0}^{N}x_kp^{k+\gamma}|_p=max(|x_0p^\gamma|,...,|x_Np^{N+\gamma}|)=p^{-\gamma}.
\end{equation*}
Thus, $|x|_p=p^{-\gamma}$.
\end{proof}

\begin{ex}
We try to find the first three terms of the $5$-adic expansion of $\frac{4}{3}$ as an example of how $p$-adic expansion works. Note that $|\frac{4}{3}|_5=1$, thus $\gamma=0$.

We use the proof of the existence of the $p$-adic expansion in Proposition 3 to find the terms of the expansion.  
We pick a rational number $q=\frac{-1}{3}$ so that $|\frac{4}{3}-\frac{-1}{3}|_5<1$.

Then, since the denominator $3$ is co-prime to $5$, there exists integers $s, t$ such that $3s+5t=1$. In fact, $s=2, t=-1$ is one solution. 

Then, $s(-1)=2(-1)=-2$. The remainder of $-2$ divided by $5$ is $3$, which is the coefficient of the first term of the expansion.

We calculate $|\frac{4}{3}-3|_5=|\frac{-5}{3}|_5=5^{-1}$ to find the second term. We multiply $\frac{-5}{3}$ by $5^{-1}$ and get $\frac{-1}{3}$ to reduce the $5$-adic norm to 1.

Again, we pick a rational number $\frac{-2}{1}$ so that $|-\frac{1}{3}-(-\frac{2}{1})|_5=5^{-1}<1$.Then, there exists $s',t'$ such that $s'+5t'=1$. In fact, $s=6, t=-1$ is a solution. Thus, $-2s=-12$. We find that 3, as a remainder of -12 divided by 5, is the coefficient of the second term of the expansion.

By using the same method to ($\frac{4}{3}-3-5\cdot3$), we have 1 as the coefficient of the third term of the expansion.

Thus, $$\frac{3}{4}=5^0(3+3\cdot 5+1 \cdot 5^2+\dotsb)$$ in the $5$-adic sense.
\end{ex}

\subsection{Topology on $\mathbb{Q}_p$}

This subsection will cover basic concepts of topology, the ring of $p$-adic integers $\mathbb{Z}_p$, and some important topological properties of $\mathbb{Q}_p$, such as the fact that $\mathbb{Q}_p$ is locally compact.

\begin{mydef}

Fix a normed field $(\mathbb{F}, |\cdot|)$. 

\begin{itemize}

\item[(i)] An open ball of radius $r>0$ centered at a point $x\in\mathbb{F}$ is the set $B_r(x)=\left\{ y\in\mathbb{F}: \left| y-x\right|<r \right\}$.
\item[(ii)] A closed ball of radius $r>0$ centered at a point $x\in\mathbb{F}$ is the set $B_r(x)=\left\{ y\in\mathbb{F}: \left| y-x\right|\leq r \right\}$.

\item[(iii)] Let $x\in\mathbb{F}$ and $A\subseteq\mathbb{F}$. Then $x$ is a limit point of $A$ if, for every $\epsilon>0$, $B_\epsilon(x)\cap A$ contains a point other than $x$.

\item[(iv)] A set $C\subseteq\mathbb{F}$ is called closed if it contains all of its limit points.

\item[(v)] A set $O\subseteq \mathbb{F}$ is called open if for every $x\in O$, there exists $\epsilon>0$ such that $B_\epsilon(x)\subseteq O$.

\item[(vi)] A set $A\subseteq\mathbb{F}$ is called bounded if there exists $M>0$ such that $|x|<M$ for any $x\in A$.

\end{itemize}
\end{mydef}

\begin{prop}
Fix a normed field $(\mathbb{F}, |.|)$. Every closed ball in $\mathbb{F}$ is closed.
\end{prop}

\begin{proof}
Fix $B_r(x)=\{ y\in\mathbb{F}: \left| y-x\right|\leq r \}$. Let $z$ be a limit point of $B_r(x)$. If $\epsilon>0$ fixed, we have $B_\epsilon(z)\cap B_r(x)\neq\varnothing$.

Let $n\in B_\epsilon (Z)\cap B_r(x)$. Then, 
\begin{equation*}
|z-x|=|z-y+y=x|\leq|z-y|+|y-x|\leq \epsilon+r.
\end{equation*}
As $\epsilon\rightarrow 0$, we have $|\epsilon+r|\rightarrow r$. Thus, $|z-x|\leq r$. It follows that $z\in B_r(x)$ by definition. Therefore, $B_r(x)$ is closed.
\end{proof}

\begin{prop}
Fix a normed field $(\mathbb{F}, |\cdot|)$. Every open ball in $\mathbb{F}$ is open.
\end{prop}

\begin{proof}
Fix 
\begin{equation*}
B_r(x)=\{ y\in\mathbb{F}: \left| y-x\right|< r \}
\end{equation*}
an open ball in $\mathbb{F}$.

Pick $y\in B_r(x)$. Then, $|y-x|=r'<r$ for some $r'>0$. Let $\epsilon=\frac{r-r'}{2}$. Pick y'$\in B_\epsilon(y)$. Then 
\begin{equation*}
|y-x|=|y-y'+y'-x|\leq|y-y'|+|y'-x|
<\frac{r-r'}{2}+r<r-r'+r=r.
\end{equation*}
Thus, y'$\in B_r(x)$.

We conclude that $B_\epsilon(y)\subseteq B_r(x)$. By definition, it follows that $B_r(x)$ is an open set.
\end{proof}

\begin{thm}
Any open ball in $\mathbb{Q}_p$ is a closed ball and vice versa.
\end{thm}

\begin{proof}
Fix $B_r(x)$ as an open ball in $\mathbb{Q}_p$, where $r>0$. 
\begin{equation*}
B_r(x)= \left\{ y\in\mathbb{Q}_p:\left| y-x\right|<r\right\}. 
\end{equation*}
There exists a smallest integer $k$ such that $r\leq p^{-k}$. Then, $$B_r(x)=\{ y\in\mathbb{Q}_p: \left| y-x\right|<p^{-k} \}=\{ y\in\mathbb{Q}_p: \left| y-x\right|\leq p^{-k-1} \}.$$ Thus, $B$ is a closed ball by definition.
A similar argument proves a closed ball in $\mathbb{Q}_p$ is an open ball.
\end{proof}

\begin{mydef}
A $p$-adic number which satisfies $\left| x\right|_p\leq1$ is called a $p$-adic integer. It can be easily verified that the set of $p$-adic integers form a ring. We denote the ring of $p$-adic integers by $\mathbb{Z}_p$. 
\end{mydef}

\begin{prop}
All integers are $p$-adic integers. 
\end{prop}

\begin{proof}
Fix $p$ as a prime number. Let $x\in\mathbb{Z}$. If $x$ is co-prime to $p$ we have $| x|_p = p^0 =1$. If $x$ is not co-prime to $p$, we have $\gamma(x)>0$, and $| x|_p=p^{-\gamma}>1$.
\end{proof}

\begin{mydef}
We define $p\mathbb{Z}_p$ to be the set of $p$-adic numbers that satisfy $$\left| x\right|_p\leq p^{\shortminus1},$$
and $p^2\mathbb{Z}_p$ to be the set of $p$-adic numbers that satisfy $$\left| x\right|_p\leq p^{\shortminus2}.$$ More generally, we define $p^n\mathbb{Z}_p$ to be the set of elements $x\in\mathbb{Q}_p$ such that $$\left| x\right|_p\leq p^{\shortminus{n}}$$ for any $n\in\mathbb{N}$.
\end{mydef}

\begin{remark}
It is easy to check that $$\mathbb{Z}_p\supset p\mathbb{Z}_p
\supset p^2\mathbb{Z}_p\supset...\supset p^n\mathbb{Z}_p\supset \dotsb .$$
Note that $\left| x\right|_p=1$ if $x\in\mathbb{Z}_p\setminus{p\mathbb{Z}_p}$, and
$\left| x\right|_p=p^{\shortminus1}$ if $x\in p\mathbb{Z}_p\setminus{p^2\mathbb{Z}_p}$. More generally, $\left| x\right|_p=p^{\shortminus n}$ if $x\in p^n\mathbb{Z}_p\setminus{p^{n+1}\mathbb{Z}_p}$ for some $n \in\mathbb{N}$.
\end{remark}

\begin{mydef}
Fix a normed field $(\mathbb{F}, |\cdot|)$. A topological space $X$ in $\mathbb{F}$ is called compact if every open cover of it has a finite subcover. That is, if $X\subseteq\cup_{m\in K} m$, with $K$ a collection of open sets, then there exists a finite collection $N\subseteq K$ such that $X\subseteq\cup_{n\in N} n$.
\end{mydef}

\begin{mydef}
Fix a normed field $(\mathbb{F}, |\cdot|)$. A topological space $X$ in $\mathbb{F}$ is called locally compact if each $x\in X$ has a compact neighborhood. That is, there exists an open set $Y$ and a compact set $Z$ such that $x\in Y\subseteq Z$.
\end{mydef}

\begin{thm} ($p$-adic Heine-Borel Theorem) A subset $X\subseteq\mathbb{Q}_p$ is compact if it is closed and bounded. 
\end{thm}

\begin{proof}
See Corollary 4 on page 8 of \cite{Jack}.
\end{proof}

\begin{cor}
The field $\mathbb{Q}_p$ is locally compact.
\end{cor}

\begin{proof}
Let $x\in\mathbb{Q}_p$. Consider $Z=\{ x+y: y\in\mathbb{Z}_p\}$.
We claim $Z$ is a closed ball centered at $x$. We have
\begin{equation*}
\begin{split}
&Z=\{ x+y\in\mathbb{Q}_p: |y|_p\leq1\}=\{ x+y\in\mathbb{Q}_p: |x+y-x|_p\leq1\}
\\
&=\{ z\in\mathbb{Q}_p: |z-x|_p\leq1\}.
\end{split}
\end{equation*}

Thus, $Z$ is a closed ball by definition. It is an open ball by Theorem 7 and an open set by Proposition 6. By Proposition 7, $Z$ is closed. 

Fix $k\in Z$. Then, $$|k|_p\leq max(|x|_p, |y|_p)$$ for some $y\in\mathbb{Z}_p$. 

Since $|y|_p\leq1$, and $|x|_p$ is fixed, pick $M=2 \cdot max (|x|_p,1)$. We have $|k|_p<M$ for any $k\in Z$. Thus, $Z$ is bounded. 

Since $Z$ is closed and bounded, it is compact by Theorem 8. Thus, $Z$ is a compact neighborhood of $x$, which means $\mathbb{Q}_p$ is locally compact.

\end{proof}

\begin{thm}
The ring $\mathbb{Z}_p$ is compact.
\end{thm}

\begin{proof}
We have that $\mathbb{Z}_p=\{ x\in\mathbb{Q}_p: |x|_p\leq1 \}$ is a closed ball by definition, and therefore closed by Proposition 3. Note also that $\mathbb{Z}_p$ is bounded by definition. Thus, $\mathbb{Z}_p$ is compact by Theorem 8.
\end{proof}

\section{Integration on $\mathbb{Q}_p$}

In this section, we introduce the notion of a Haar measure for locally compact topological abelian groups (such as $\mathbb{Q}_p$). This will allow us to set up the concept of an integral of a complex-valued function of a $p$-adic variable. We then present basic rules for $p$-adic calculus using the properties of the Haar measure, $p$-adic numbers and knowledge of abstract algebra.

\begin{mydef}
(Left, right Haar Measure)

Let $(G,\cdot)$ be a topological group. A left Haar measure on $G$ is a nonzero regular Borel measure $\mu$ on $G$ such that $\mu(x\cdot E)=\mu(E)$ for all $x\in G$ and all measurable subsets $E$ of $G$. A right Haar measure on $G$ is a nonzero regular Borel measure $\mu$ on $G$ such that $\mu(E\cdot x)=\mu(E)$ for all $x\in G$ and all measurable subsets $A$ of $G$.
\end{mydef}

\begin{thm}
Let $(G,\cdot)$ be a locally compact topological group. Then, there exists a unique left (right) Haar measure up to scaling by a positive constant.
\end{thm}

\begin{proof}
See the proof of Theorem 4.3 on page 6, and Theorem 4.6 on page 11 of \cite{Gleason}.
\end{proof}

\begin{thm}
Let $(G,\cdot)$ be a locally compact topological abelian group. Then, the left Haar measure is also the right Haar measure.
\end{thm}

\begin{proof}
Fix $(G,\cdot)$ a locally compact abelian topological group.

By Theorem 10 and the definition of left and right Haar measures, there exists a left and a right Haar measure $\mu$, $\mu'$ respectively. Fix $E$ a measurable subset of $G$, and $x\in G$. Then, $\mu(E\cdot x)=\mu(E)$, $\mu'(E)=\mu'(x\cdot E)$.

Since $G$ is abelian, $\mu'(E)=\mu'(E\cdot x)$. By definition, this right Haar measure is the left Haar measure.
\end{proof}

By the previous two theorems, we can define the unique Haar measure on any locally compact topological abelian group, since the left Haar measure exists and is unique, and the left Haar measure is also the right Haar measure.

\begin{mydef} (Haar Measure)

Let $(G,\cdot)$ be a locally compact topological abelian group. The nonzero regular Borel measure $\mu$ on $G$ such that $\mu(x\cdot E)=\mu(E)$ for all $x\in G$ and all measurable subsets $E$ of $G$ is called the Haar Measure of $G$.
\end{mydef}

\begin{lem}
$(\mathbb{Q}_p, +)$ is a locally compact topological abelian group. In addition, there exists a unique Haar measure $dx$ such that $\int_{\mathbb{Z}_p}dx=1$, and $d(x+a)=dx$ for any $x, a \in\mathbb{Q}_p$.
\end{lem}

\begin{proof}
By Corollary 1 in Section 2.3, $\mathbb{Q}_p$ is locally compact. Fix $x,y\in\mathbb{Q}_p$. We have $|x+y|_p=|y+x|_p$. Thus, $(\mathbb{Q}_p, +)$ is abelian. The proof that $(\mathbb{Q}_p, +)$ is a topological group can be seen under Exercise 3.5 on page 9 of \cite{Vladimirov}.

By Theorem 10 and Theorem 11, $(\mathbb{Q}_p, +)$ has a unique Haar measure $dx$ up to scaling by a positive constant such that $d(x+a)=dx$ for any $x \in\mathbb{Q}_p$ and $a>0, a\in\mathbb{Q}_p$. We normalize the measure by setting $\int_{\mathbb{Z}_p}dx=1$. Then $dx$ is unique.
\end{proof}

\begin{thm}
Let $dx$ be the Haar measure for $(\mathbb{Q}_p, +)$. Then, $$d(ax)=\left| a\right|_p dx$$ for all a$\in\mathbb{Q}_p^*$, i.e. $\mathbb{Q}_p\setminus \{0\}$.
\end{thm}

\begin{proof}
This proof is inspired by (3.7) on page 13 of \cite{Zuniga}. 

\vspace{0.3cm}

We want to equivalently show that $\int_{aU}dx=|a|_p\int_U dx$ for any Borel set $U\subseteq\mathbb{Q}_p$. 

Fix a$\in\mathbb{Q}_p^*$. Then U$\mapsto\int_{aU}dx$ is a Haar measure for ($\mathbb{Q}_p$,+). Hence, there exists a unique positive constant C such that $\int_{aU}dx=C\int_Udx$ for any Borel set $U\subseteq\mathbb{Q}_p$. 

We fix $U=\mathbb{Z}_p$ to calculate C.
Assume $a\in\mathbb{Z}_p$. Then $a=p^m u$, where $m\in\mathbb{N}$, and $u \in\mathbb{Z}_p^*$.

We have
\begin{align*}
\int_{\mathbb{Z}_p} 1 \text{d}x &=\sum_{b\in\mathbb{Z}_p /p^m \mathbb{Z}_p} \int_{b+p^m \mathbb{Z}_p} \text{d}x \\
&= \sum_{b\in\mathbb{Z}_p/p^m \mathbb{Z}_p} \int_{p^m\mathbb{Z}_p} \text{d}x \\
&= p^m \int_{p^m\mathbb{Z}_p} \text{d}x =1
\end{align*}

Note that the first equation derives from the fact that $$\mathbb{Z}_p=\bigcup_{b\in\mathbb{Z}_p/p^m\mathbb{Z}_p}(b+p^m\mathbb{Z}_p).$$ The second equation is based on Lemma 1. The term $p^m$ in the third equality represents the number of distinct cosets of $p^m\mathbb{Z}_p$ in $\mathbb{Z}_p$, since the $p$-adic expansion of the form $b_0+b_1p+b_2p^2+...+b_{m-1}p^{m-1}+p^m\mathbb{Z}_p$ has $p$ choices for each $b_k$, with $k$ ranging from $0$ to $m-1$.

Thus, 
\begin{equation*}
|a|_p=p^{-m}=\int_{a\mathbb{Z}_p}\text{d}x.
\end{equation*}

If $x\in\mathbb{Q}_p\setminus\mathbb{Z}_p$, we can get the same result by setting $a=p^{-m}u$, and $b\in{p^{-m} \mathbb{Z}_p}/\mathbb{Z}_p$. 
Thus, $C=|a|_p$, and $d(ax)=|a|_p dx$.
\end{proof}

\begin{thm}
For any m$\in\mathbb{Z}$
$$\mu(p^m \mathbb{Z}_p) = p^{-m}.$$
\end{thm}

\begin{proof}
Fix m$\in\mathbb{Z}$. Then 
\begin{align*}
\mu(p^m \mathbb{Z}_p)&=\int_{x\in p^m \mathbb{Z}_p}dx=\int_{y\in\mathbb{Z}_p}d(p^m y) \\
&=\int_{y\in\mathbb{Z}_p} |p^m|_p dy \\
&=p^{-m}\int_{y\in\mathbb{Z}_p}dy \\
&=p^{-m}.
\end{align*}

\end{proof}

\begin{cor}
For any $\gamma \in \mathbb{Z}$, we have
$$\int_{{\left| x\right|_p}=p^\gamma}dx=p^\gamma (1-\frac{1}{p}).$$
\end{cor}

\begin{proof}
\begin{equation*}
\int_{\left| x\right|_p}dx=\int_{\left| x\right|_p\leq p^\gamma}dx-\int_{\left| x\right|_p\leq p^{\gamma-1}}dx=p^\gamma-p^{\gamma-1}
\end{equation*}
 by Theorem 13.
\end{proof}

\begin{cor}
For any $s > -1,\  s\in\mathbb{Z}$, we have
$$\int_{\mathbb{Z}_p} |x|_p^s =\frac{p-1}{p-p^{\shortminus s}}.$$ 
\end{cor}

\begin{proof}
For $s>-1$ we have
\begin{align*}
&\int_{\mathbb{Z}_p}\left| x\right|_p^sdx 
=\sum^{\infty}_{\gamma=0}\int_{\left| x\right|_p=p^{\shortminus{\gamma}}} p^{\shortminus{\gamma}s}dx
\\&\quad
=\sum^{\infty}_{\gamma=0} p^{\shortminus{\gamma}s}(p^{-\gamma}-p^{-{\gamma-1}})
=\frac{p-1}{p-p^{\shortminus s}}
\end{align*}
using the geometric series formula.
\end{proof}

\begin{cor}
When $s<-1$ we have
$$\int_{\mathbb{Q}_p \setminus \mathbb{Z}_p} |x|_p^s= -\frac{p-1}{p-p^{\shortminus s}}.$$ 
\end{cor}

\begin{proof}
Fix $s<-1$. Then
\begin{align*}
\int_{\mathbb{Q}_p\setminus{\mathbb{Z}_p}}\left| x\right|_p^s dx=\sum^{\infty}_{\shortminus\gamma=1}\int_{\left| x\right|_p=p^{\shortminus{\gamma}}} p^{\shortminus{\gamma}s}dx=\sum^{\infty}_{\shortminus\gamma=1}p^{\shortminus{\gamma s}} (p^{-\gamma}-p^{-{\gamma-1}}).
\end{align*}
This sum equals $-\frac{p-1}{p-p^{\shortminus s}}$ by the geometric series formula. 
\end{proof}

Theorem 12, Theorem 13, and Corollaries 2, 3, and 4 give a good foundation for the rules of $p$-adic calculus, and are the basic tools for solving $p$-adic integration problems in the following sections.

\section{A $p$-adic derivative operator}

This is a brief section introducing the $p$-adic derivative operator and its two key properties.
\begin{mydef}
Let $\alpha>0$. The $p$-adic derivative operator $D^{\alpha}$ is defined as 
\begin{equation}
D^\alpha f(x)=\frac{p^\alpha -1}{1-p^{-\alpha-1}} \int_{\mathbb{Q}_p} \frac{f(x)-f(y)}{\left| x-y\right|_p^{\alpha+1}}dy
\end{equation}
for any complex-valued function $f(x)$ with variables in the $p$-adic field $\mathbb{Q}_p$. We define the $p$-adic Gamma function to be the function
\begin{equation}
\Gamma_p(x)=\frac{1-p^{x-1}}{1-p^{\shortminus x}}
\end{equation}
for $x\neq0$. Then, 
\begin{equation}
D^\alpha f(x)=\frac{1}{\Gamma_p  ({\shortminus \alpha})} \int_{\mathbb{Q}_p}  \frac{f(y)-f(x)}{\left| x-y\right|_p^{\alpha+1}}dy.
\end{equation}

\end{mydef}

There are two important properties of the differential operator which are stated in the following proposition.
\begin{prop}
\begin{itemize}
\item[(1)] For $\alpha, \beta>0$, we have $$D^\alpha [D^\beta f(x)]=D^{\alpha+\beta} f(x)$$

\item[(2)] For $n \in \mathbb{N}$, $\alpha>0$, we have
\begin{equation*}
D^\alpha {\left| x\right|_p}^n=\frac{\Gamma_p(n+1)}{\Gamma_p(n-\alpha+1)} \left| x\right|_p^{n-\alpha}.
\end{equation*}
\end{itemize}
\end{prop}

\begin{proof}
We prove the second identity because its proof gives readers an idea about how $p$-adic integration looks like and helps with the understanding of a similar calculation in Section 5 of the thesis. 

\vspace{0.4cm}

We have

\begin{align*}
&D^\alpha {\left| x\right|_p}^n=\frac{1}{\Gamma_p  ({\shortminus \alpha})} \int_{\mathbb{Q}_p}  \frac{{\left| y\right|_p}^n -\left| x\right|_p^n}{\left| x-y\right|_p^{\alpha+1}}dy
\\&\quad
=\frac{1}{\Gamma_p(-\alpha)} \left(\int_{|x|_p<|y|_p} \frac{{\left| y\right|_p}^n -\left| x\right|_p^n}{\left| x-y\right|_p^{\alpha+1}}dy+0+\int_{|x|_p>|y|_p} \frac{{\left| y\right|_p}^n -\left| x\right|_p^n}{\left| x-y\right|_p^{\alpha+1}}dy\right).
\end{align*}

Set $|x|_p=p^t$, where $t$ is some integer.
Then, 
\begin{align*}
&D^\alpha {\left| x\right|_p}^n
=\frac{1}{\Gamma_p(-\alpha)} \left(\int_{|x|<|y|} \frac{|y|_p^n-p^{tn}}{|y|_p^{\alpha+1}}dy+\int_{|x|>|y|} \frac{|y|_p^n-p^{tn}}{p^{t(\alpha+1)}}dy\right)
\\&\quad
=\frac{1}{\Gamma_p(-\alpha)} \left(\sum_{k=t+1}^\infty \frac{p^{kn}-p^{tn}}{p^{k(\alpha+1)}}p^k \left(1-\frac{1}{p}\right)+\sum_{-\infty}^{k=t-1}\frac{p^{kn}-p^{tn}}{p^{t(\alpha+1)}}p^k\left(1-\frac{1}{p}\right)\right).
\end{align*}

The left series converges only if $n<\alpha$, the right series converges only if $n+1>0$. However, we can appropriately analytically continue the two series.

Then after calculation, 
\begin{equation*}
D^\alpha {\left| x\right|_p}^n=\frac{\Gamma_p(n+1)}{\Gamma_p(n-\alpha+1)} \left| x\right|_p^{n-\alpha}.
\end{equation*}
\end{proof}

\section{The $p$-adic Sch\"{o}dinger equation}

\subsection{Classical Schr\"{o}dinger equation}

We begin by introducing the classical time-independent Schr\"{o}dinger equation and its solutions following results in Chapter 2 of \cite{Griffiths}.

\begin{mydef}
A time-independent Schr\"{o}dinger equation is an equation of the form
\begin{equation}
-\frac{h^2}{2m}\frac{d^2\Psi}{dx^2}+\frac{1}{2}m\omega^2x^2\Psi=E\Psi, 
\end{equation}
where $h,m,\omega, E$ are some complex constants. Here $E$ is called the energy and satisfies $E\geq0$, and $x$ is the variable of the function $\Psi(x)$.
\end{mydef}

Fix a time-independent Schr\"{o}dinger equation. It can be rewritten as:
\begin{equation}
\frac{1}{2m}\left[\left(\frac{h}{i}\frac{d}{dx}\right)^2+(m\omega x)^2 \right]\Psi=E\Psi
\end{equation}

The left side of the above equation inspires us to think about the factorization of complex numbers, $x^2+y^2=(x-iy)(x+iy)$. We consider 
\begin{equation}
a_+=\frac{1}{\sqrt{2m}} \left(\frac{h}{i}\frac{d}{dx}+im\omega x\right)
\end{equation}
\begin{equation}
a_-=\frac{1}{\sqrt{2m}}\left(\frac{h}{i}\frac{d}{dx}-im\omega x\right).
\end{equation}

We wish to explore the operators $a_-a_+$ and $a_+a_-$ using an arbitrary test function $f(x)$.

After a calculation we have 
\begin{equation}
a_-a_+=\frac{1}{2m}\left[\left(\frac{h}{i}\frac{d}{dx}\right)^2+\left(m\omega x\right)^2 \right]+\frac{1}{2}h\omega
\end{equation}

\begin{equation}
a_+a_-=\frac{1}{2m}[(\frac{h}{i}\frac{d}{dx})^2+(m\omega x)^2]-\frac{1}{2}h\omega
\end{equation}

By equations (8), (11), (12), we have:
\begin{equation}
(a_-a_+-\frac{1}{2}h\omega)\Psi=E\Psi
\end{equation}

\begin{equation}
(a_+a_-+\frac{1}{2}h\omega)\Psi=E\Psi.
\end{equation}

\begin{remark}

Equation (13) and Equation(14) are just another form of the

Schr\"{o}dinger equation in (7).
\end{remark}

\begin{thm}
Fix a time-independent Schr\"{o}dinger equation 
\begin{equation*}
-\frac{h^2}{2m}\frac{d^2\Psi}{dx^2}+\frac{1}{2}m\omega^2x^2\Psi=E\Psi, 
\end{equation*}
where h,m,$\omega$, $E$ are some complex constants, $E$ denotes the energy, and $x$ is the variable of the function $\Psi(x)$. 
Let $a_+, a_-$ be the operators in (9) and (10). If $\Psi_0$ satisfies the Schrodinger Equation, then:

\begin{itemize}
\item[(a)] $a_+\Psi_0$ is a solution to the new Schr\"{o}dinger equation with energy $(E+h\omega)$

\item[(b)] $a_-\Psi_0$ is a solution to the new Schrodinger Equation with energy $(E-h\omega)$.
\end{itemize}
\end{thm}

\begin{proof}
It is sufficient to prove (a), since (b) can be proved similarly.

We claim $(a_+a_-+\frac{1}{2}h\omega)(a_+\Psi)=(E+h\omega)(a_+\Psi)$, which is equivalent to (a). Indeed,

\begin{align*}
&(a_+a_-+\frac{1}{2}h\omega)(a_+\Psi)=(a_+a_-a_++\frac{1}{2}h\omega a_+)\Psi
\\&\quad
=a_+(a_-a_++\frac{1}{2}h\omega)\Psi=a_+[(a_-a_+-\frac{1}{2}h\omega)\Psi+h\omega\Psi]
\\&\quad
=a_+(E\Psi+h\omega\Psi)=(E+h\omega)(a_+\Psi).
\end{align*}
\end{proof}

By Theorem 14, we find $a_+, a_-$ are operators that help us find solutions to the Schr\"{o}dinger equation with raising or lowering energy. Therefore, we make the following definition.

\begin{mydef}
For a time-independent Schr\"{o}dinger equation
\begin{equation*}
-\frac{h^2}{2m}\frac{d^2\Psi}{dx^2}+\frac{1}{2}m\omega^2x^2\Psi=E\Psi, 
\end{equation*}
$a_+$ defined in (9) is called the raising operator of the equation, and $a_-$ defined in (10) is called the lowering operator.
\end{mydef}

\begin{remark}
Notice that if the raising operator is iterated, the energy will be lower and lower, reaching a level $E<0$, which contradicts Definition 17. Therefore, a minimum energy level $E_0$ and $\Psi_0$ are required. Its existence and form are proved and stated in the following theorem.
\end{remark}

\begin{thm}
There exists $\Psi_0$ such that $a_-\Psi_0=0$ for a Schr\"{o}dinger equation and the lowering operator $a_-$. In this case, 
\begin{equation}
\Psi_0(x)=A_0e^{-\frac{m\omega}{2h}x^2},
\end{equation} where $A_0$ is some constant.
The corresponding energy is
\begin{equation}
E_0=\frac{1}{2}h\omega.
\end{equation}

\end{thm}

\begin{proof}
The existence part is shown on Chapter 2, page 34 of \cite{Griffiths}. We will find the explicit expression of $\Psi_0(x)$.

Given $a_-\Psi_0=0$, we have 
\begin{equation*}
\frac{1}{\sqrt{2m}} \left(\frac{h}{i}\frac{d\Psi_0}{dx}-im\omega x\Psi_0 \right)=0.
\end{equation*}

After simplification, we have the following differential equation: 
\begin{equation*}
\frac{d\Psi_0}{dx}+\frac{m\omega}{h}x\Psi_0=0.
\end{equation*}

Using the formula for solving first-order differential equations, let 
$$I(x)=e^{\int \frac{m\omega}{h}x},\  Q(x)=0.$$ 
Then 
\begin{align*}
&\Psi_0(x)=\frac{1}{I(x)}\left[\int I(x)Q(x)dx+A_0\right]=
\\&\quad
\frac{1}{e^{\frac{m\omega x^2}{2h}}}[0+A_0]=A_0e^{-\frac{m\omega}{2h}x^2}.
\end{align*}

We replace the expression of $\Psi_0(x)$ into equation (16), and have 

\begin{align*}
a_+(a_-\Psi_0(x))+\frac{1}{2}h\omega\Psi_0(x)=E_0\Psi_0(x).
\end{align*}
Given $a_-\Psi_0=0$, we have $a_+(a_-\Psi_0(x))=0$, and therefore $E_0=\frac{1}{2}h\omega$. 
\end{proof}

It is easy to find explicit expressions of solutions to the Schr\"{o}dinger equation with higher energy using the raising operator $a_+$ based on $\Psi_0$ and $E_0$, and thus we have the following definition.

\begin{mydef}
Given $\Psi_0$ and $E_0$ for a time-independent Schr\"{o}dinger equation as in Theorem 18, we define $\Psi_n(x)=A_n(a_+)^ne^{-\frac{m\omega}{2h}x^2}$ as solutions to the equation with energy $E_n=(n+\frac{1}{2})h\omega$, which represents the $n^{th}$ excited state.
\end{mydef}

In the next section, we will explore a $p$-adic analog of the Schr\"{o}dinger equation inspired by the explicit expression of solutions, and raising and lowering operators we find in this section.

\subsection{A $p$-adic Schr\"{o}dinger equation}

We will prove Theorem 16 and 17 to prepare for exploring the $p$-adic analogs of the Schr\"{o}dinger equation since they are useful tools to simplify the calculation.
\begin{thm}
Fix $n\in \mathbb{N}$. Let 
\begin{equation}
f_n(x) = \begin{cases} |x|_p^n &\text{ if } x \in \mathbb{Z}_p  \\ 0 & \text{ otherwise} \end{cases}.
\end{equation}

Then,
\begin{equation}
D^{\alpha} f_n (x) = \begin{cases} \frac{\Gamma_p(n+1)}{\Gamma_p(n-1)}|x|_p^{n-\alpha}+\frac{1}{\Gamma_p(-\alpha)}\frac{p-1}{p-p^{-n+\alpha+1}} &\text{ if } x \in \mathbb{Z}_p  \\ \frac{1}{\Gamma_p(-\alpha)}\left| x\right|_p^{\shortminus(\alpha+1)}\frac{p-1}{p-p^{\shortminus n}} & \text{ otherwise} \end{cases}.
\end{equation}
\end{thm}

\begin{proof}

Let $x\in\mathbb{Z}_p$, and let $|x|_p=p^t$ for some integer $t\leq0$. We have

\begin{align*}
& D^\alpha f_n(x)=\frac{1}{\Gamma_p(-\alpha)}\int_{\mathbb{Q}_p}\frac{f_n(y)-f_n(x)}{\left| y-x\right|_p^{\alpha+1}}dy \\&\quad
=\frac{1}{\Gamma_p(-\alpha)} \left(\int_{|y|_p>|x|_p}\frac{|y|_p-|x|_p}{|y|_p^{\alpha+1}}+0+\int_{|y|_p<|x|_p}\frac{|y|_p-|x|_p}{|x|_p^{\alpha+1}} \right) \\&\quad
=\frac{1}{\Gamma_p(-\alpha)}\left(\sum_{k=t+1}^0 \frac{p^{kn}-p^{tn}}{p^{k(\alpha+1)}}(p^k-p^{k-1})+\sum_{k=-\infty}^{t-1} \frac{p^{kn}-p^{tn}}{p^{t(\alpha+1)}}(p^k-p^{k-1})\right) \\&\quad
=\frac{\Gamma_p(n+1)}{\Gamma_p(n-1)}|x|_p^{n-\alpha}+\frac{1}{\Gamma_p(-\alpha)}\frac{p-1}{p-p^{-n+\alpha+1}},
\end{align*}
after some calculation.

\vspace{0.4cm}

Let $x\in\mathbb{Q}_p\setminus\mathbb{Z}_p$. Then
\begin{align*}
&D^\alpha f_n(x)=\frac{1}{\Gamma_p(-\alpha)}\int_{\mathbb{Q}_p}\frac{f_n(y)-0}{\left| y-x\right|_p^{\alpha+1}}dy
=\frac{1}{\Gamma_p(-\alpha)} \left(\int_{\mathbb{Z}_p}\frac{\left| y\right|_p^n}{\left| x\right|_p^{\alpha+1}}+\int_{\mathbb{Q}_p\setminus\mathbb{Z}_p}\frac{0-0}{\left| y-x\right|_p^{\alpha+1}}\right)
\\&\quad
=\frac{1}{\Gamma_p(-\alpha)}\left| x\right|_p^{\shortminus(\alpha+1)}\frac{p-1}{p-p^{\shortminus n}}.
\end{align*}

\end{proof}

\begin{thm}
Fix $n\in\mathbb{N}$. Let 
\begin{equation}
g_n(x) = \begin{cases} 0 &\text{ if } x \in \mathbb{Z}_p  \\ |x|_p^n & \text{ otherwise} \end{cases}.
\end{equation}
Then \begin{equation}
D^{\alpha} g_n (x) = \begin{cases} \frac{1}{\Gamma_p(-\alpha)}(-\frac{p-1}{p-p^{\shortminus(n-\alpha-1)}}) &\text{ if } x \in \mathbb{Z}_p  

\\ \frac{\Gamma_p(n+1)}{\Gamma_p(n-1)}|x|_p^{n-\alpha}-\frac{1}{\Gamma_p(-\alpha)}\frac{p-1}{p-p^{-n}}|x|_p^{-\alpha-1} & \text{ otherwise} 

\end{cases}.
\end{equation}
\end{thm}

\begin{proof}

Let $x\in\mathbb{Z}_p$. We have

\begin{align*}
D^\alpha g_n(x) &=\frac{1}{\Gamma_p(-\alpha)}\int_{\mathbb{Q}_p}\frac{g_n(y)-g_n(x)}{\left| y-x\right|_p^{\alpha+1}}dy \\&\quad
=\frac{1}{\Gamma_p(-\alpha)}\int_{\mathbb{Q}_p}\frac{g_n(y)}{\left| y-x\right|_p^{\alpha+1}}dy \\&\quad
=\frac{1}{\Gamma_p(-\alpha)}\left(\int_{\mathbb{Z}_p} 0dy+\int_{\mathbb{Q}_p\setminus\mathbb{Z}_p}\frac{\left| y\right|_p^n}{\left| y\right|_p^{\alpha+1}}dy\right) \\&\quad
=\frac{1}{\Gamma_p(-\alpha)}\left(-\frac{p-1}{p-p^{\shortminus(n-\alpha-1)}}\right).
\end{align*}

Let $x\in\mathbb{Q}_p\setminus\mathbb{Z}_p$. We write $|x|_p=p^t$ for some integer $t>0$. Then

\begin{align*}
D^\alpha g_n(x)&=\frac{1}{\Gamma_p(-\alpha)}\int_{\mathbb{Q}_p}\frac{g_n(y)-\left| x\right|_p^n}{\left| y-x\right|_p^{\alpha+1}}dy \\&\quad
=\frac{1}{\Gamma_p(-\alpha)} \Big(\int_{|y|_p>|x|_p}\frac{|y|_p^n-|x|_p^n}{|y|_p^{\alpha+1}} dy+0+\int_{|y|_p<|x|_p,\ y\in\mathbb{Z}_p}\frac{-|x|_p^n}{|x|_p^{\alpha+1}} dy \\&\quad
+\int_{|y|_p<|x|_p,\  y\in\mathbb{Q}_p\setminus\mathbb{Z}_p} \frac{|y|_p^n-|x|_p^n}{|x|_p^{\alpha+1}} dy \Big) \\&\quad
=\frac{1}{\Gamma_p(-\alpha)}\Big(\sum_{k=t+1}^{\infty}\frac{p^{kn}-p^{tn}}{p^{(\alpha+1)k}}(p^k-p^{k-1})+
\\&\quad
\sum_{-\infty}^{k=0}\frac{-p^{tn}}{p^{(\alpha+1)t}}(p^k-p^{k-1})
+\sum_{k=1}^{k=t-1}\frac{p^{kn}-p^{tn}}{p^{(\alpha+1)t}}(p^k-p^{k-1})\Big) \\&\quad
=\frac{\Gamma_p(n+1)}{\Gamma_p(n-1)}|x|_p^{n-\alpha}-\frac{1}{\Gamma_p(-\alpha)}\frac{p-1}{p-p^{-n}}|x|_p^{-\alpha-1}.
\end{align*}

\end{proof}

In the next few examples, we will explore solutions to the $p$-adic analog of the time-independent Schr\"{o}dinger equation.

\begin{ex}
For the $p$-adic differential equation \begin{equation}
D^2\Psi(x)+B\left| x\right|_p^2\Psi(x)=E\Psi(x),
\end{equation}
we seek solutions of the form $$\Psi_0(x)=\sum_{n=0}^\infty b_{2n}\left| x\right|_p^{2n},$$ with $x\in \mathbb{Q}_p$. Here $D^2$ is the $p$-adic differential operator introduced in Section 4. 
 
 We want to determine coefficients $b_{2n},$ $B$ and $E$, which are assumed to be real numbers, that satisfy the equation and make the series $\sum_{n=0}^\infty b_{2n}\left| x\right|_p^{2n}$ converge.

Knowing $D^2\Psi_0(x)=b_{2n}\frac{\Gamma_p(2n+1)}{\Gamma_p(2n-1)}\left| x\right|_p^{2n-2}$, where $\Gamma_p(x)=\frac{1-p^{x-1}}{1-p^{\shortminus x}}$, substituting in the equation, we can write
\begin{equation}
\sum_{n=1}^\infty b_{2n}\frac{\Gamma_p(2n+1)}{\Gamma_p(2n-1)}\left| x\right|_p^{2n-2}+B\sum_{n=0}^\infty b_{2n}\left| x\right|_p^{2n+2}=E\sum_{n=0}^\infty b_{2n}\left| x\right|_p^{2n}.
\end{equation}

For the coefficient of $\left| x\right|_p^0$ we have
\begin{equation*}
b_2\frac{\Gamma_p(3)}{\Gamma_p(1)}=Eb_0.
\end{equation*}
 Since $\Gamma_p(1)=0$, we let $E=b_2=0$ to avoid the meaningless case that the denominator is 0. We assume $b_0=1$.

For the coefficient of $\left| x\right|_p^2$ we have
\begin{equation*}
b_4\frac{\Gamma_p(5)}{\Gamma_p(3)}+Bb_0=Eb_1=0.
\end{equation*}
Then, $b_4=-B\frac{\Gamma_p(3)}{\Gamma_p(5)}$.

For the coefficient of $\left| x\right|_p^4$ we have
\begin{equation*}
b_6\frac{\Gamma_p(7)}{\Gamma_p(5)}+Bb_2=0.
\end{equation*}
Then, $b_6=-\frac{Bb_2\Gamma_p(5)}{\Gamma_p(7)}=0$. For the same reason, $b_8=\frac{B^2\Gamma_p(3)\Gamma_p(7)}{\Gamma_p(5)\Gamma_p(9)}.$

More generally, if we assume $b_0=1$,
\begin{align}
&b_{4n}=\frac{(-B)^n\Gamma_p(4n-1)...\Gamma_p(3)}{\Gamma_p(4n+1)...\Gamma_p(5)}
\\&\quad
b_{4n+2}=0
\end{align}
for $n\in\mathbb{N}$.

Next, we check whether $\sum_{n=0}^\infty$ b$_{2n}$$\left| x\right|_p^{2n}$ converges in the real sense. Notice that 
\begin{equation*}
\sum_{n=0}^\infty b_{2n}\left| x\right|_p^{2n}=\sum_{n=1}^\infty (-B)^n\frac{\Gamma_p(4n-1)...\Gamma_p(3)}{\Gamma_p(4n+1)...\Gamma_p(5)}\left| x\right|_p^{4n}.
\end{equation*}
We apply the Ratio Test and find
\begin{align*}
& lim_{n\longrightarrow\infty}\left| \frac{b_{4n+4}\left| x\right|_p^{4n+4}}{b_{4n}\left| x\right|_p^{4n}}\right|=lim_{n\longrightarrow\infty}\left| \frac{B\Gamma_p(4n+3)\left| x\right|_p^4}{\Gamma_p(4n+5)}\right|
\\&\quad
=B\left| x\right|_p^{4n}lim_{n\longrightarrow\infty}\left| \frac{\Gamma_p(4n+3)}{\Gamma_p(4n+5)}\right|.
\end{align*}

We compute $lim_{x\rightarrow\infty}\left| \frac{\Gamma_p(x)}{\Gamma_p(x+1)}\right|$, then $lim_{n\rightarrow\infty} \left| \frac{\Gamma_p(n)}{\Gamma_p(n+2)}\right|$ to find
$lim_{n\rightarrow\infty}\left| \frac{\Gamma_p(4n+3)}{\Gamma_p(4n+5)}\right|$.
We have
\begin{align*}
& lim_{x\longrightarrow\infty}\left| \frac{\Gamma_p(x)}{\Gamma_p(x+1)}\right|=lim_{x\longrightarrow\infty}\left| \frac{(1-p^{x-1})(1-p^{-(x+1)})}{(1-p^{-x})(1-p^x)}\right|=
\\&\quad
lim_{x\longrightarrow\infty}\left| \frac{(\frac{1}{p}-p^{x-1}+1-\frac{1}{p})(\frac{1}{p}-p^{-(x+1)}+1-\frac{1}{p})}{(1-p^x)(1-p^{-x})}\right|=
\\&\quad
lim_{x\longrightarrow\infty}\left(\frac{1}{p}+\frac{1-\frac{1}{p}}{1-p^x}\right)\left(\frac{1}{p}+\frac{1-\frac{1}{p}}{1-p^{-x}}\right)
=\left(\frac{1}{p}+0 \right)\left(\frac{1}{p}+1-\frac{1}{p}\right)=\frac{1}{p}.
\end{align*}
Thus
\begin{equation*}
lim_{n\rightarrow\infty}\left| \frac{\Gamma_p(n)}{\Gamma_p(n+2)}\right|=lim_{n\longrightarrow\infty}\left| \frac{\Gamma_p(n)}{\Gamma_p(n+1)}\right|lim_{n\longrightarrow\infty}\left| \frac{\Gamma_p(n+1)}{\Gamma_p(n+2)}\right|=\frac{1}{p^2}.
\end{equation*}
Consequently, $(\left| \frac{\Gamma_p(4n+3)}{\Gamma_p(4n+5)}\right|)$ as a subsequence of $\left| \frac{\Gamma_p(n)}{\Gamma_p(n+2)}\right|$ also converges to $\frac{1}{p^2}$.

In conclusion, whenever 
\begin{equation*}
lim_{n\longrightarrow\infty}\left| \frac{b_{4n+4}\left| x\right|_p^{4n+4}}{b_{4n}\left| x\right|_p^{4n}}\right|=B\left| x \right|_p^4\frac{1}{p^2}<1,
\end{equation*}
$\sum_{n=0}^\infty b_{2n}\left| x\right|_p^{2n}$ passes the ratio test and converges. However, it does not converges everywhere, since $\left| x \right|_p$ can be infinitely large. 
\end{ex}

In the next example,  we construct a solution $\Psi_0$ whose series expansion is likely to converge for any $x\in\mathbb{Q}_p$.

\begin{ex}
We guess 
\begin{align}
\psi_0(x) &= \sum_{n=0}^{\infty} c_n f_n (x) + \sum_{n=1}^{\infty} k_n g_{-n}(x) \\
&= \begin{cases} \sum_{n=0}^{\infty} c_n |x|_p^n & \text{ if } x \in \mathbb{Z}_p \\ \sum_{n=1}^{\infty} k_n |x|_p^{-n} & \text{ otherwise}. \end{cases}
\end{align}
is a solution to the equation $D^2\Psi(x)+B\left| x\right|_p^2\Psi(x)=E\Psi(x)$, where $c_n,k_n$ are some real sequences.

Let $x\in\mathbb{Z}_p$. Then, 

\begin{equation}
\begin{split}
&D^2\Psi_0(x)=\sum_{n=0}^{\infty}c_n \left(\frac{\Gamma_p(n+1)}{\Gamma_p(n-1)}|x|_p^{n-2}+\frac{1}{\Gamma_p(-2)}\frac{p-1}{p-p^{-n+3}}\right)
\\
&+\sum_{n=1}^{\infty}k_n\frac{1}{\Gamma_p(-2)}\frac{1-p}{p-p^{n+3}}. 
\end{split}
\end{equation}

We replace expression (27) into the equation and have

\begin{equation}
\begin{split}
&\sum_{n=0}^{\infty}c_n \left(\frac{\Gamma_p(n+1)}{\Gamma_p(n-1)}|x|_p^{n-2}+\frac{1}{\Gamma_p(-2)}\frac{p-1}{p-p^{-n+3}}\right)
\\
&+\sum_{n=1}^{\infty}k_n\frac{1}{\Gamma_p(-2)}\frac{1-p}{p-p^{n+3}}+B\sum_{n=0}^{\infty}c_n|x|_p^{n+2}
=E\sum_{n=0}^{\infty}c_n|x|_p^n.
\end{split}
\end{equation}

For the coefficient of $|x|_p^{-2}$ we have 
\begin{equation*}
c_0\frac{\Gamma_p(1)}{\Gamma_p(-1)}=0.
\end{equation*}
Since $\Gamma_p(1)=0$, we do not have information about $c_0$.

For the coefficient of $|x|_p^{-1}$ we have
\begin{equation*}
c_1\frac{\Gamma_p(2)}{\Gamma_p(0)}=0.
\end{equation*}
We let $c_1=0$ since $\Gamma_p(0)$ is undefined.

For the coefficient of $|x|_p^0$ we have
\begin{equation*}
c_2\frac{\Gamma_p(3)}{\Gamma_p(1)}+\sum_{n=0}^{
\infty}c_n\frac{1}{\Gamma_p(-2)}\frac{p-1}{p-p^{-n+3}}+\sum_{n=1}^{\infty}k_n\frac{1}{\Gamma_p(-2)}\frac{1-p}{p-p^{n+3}}=Ec_0.
\end{equation*}
Since $\Gamma_p(1)=0$, we set $c_2=0$. Then,
\begin{equation}
\sum_{n=0}^{\infty}c_n\frac{1}{\Gamma_p(-2)}\frac{p-1}{p-p^{-n+3}}+\sum_{n=1}^{\infty}k_n\frac{1}{\Gamma_p(-2)}\frac{1-p}{p-p^{n+3}}=Ec_0.
\end{equation}

For the coefficient of $|x|_p^1$ we have 
\begin{equation*}
c_3\frac{\Gamma_p(4)}{\Gamma_p(2)}=Ec_1.
\end{equation*}
Thus, $c_3=0$.

For the coefficient of $|x|_p^2$ we have
\begin{align*}
& c_4\frac{\Gamma_p(5)}{\Gamma_p(3)}+Bc_0=Ec_2  
\\&\quad c_4=\frac{-Bc_0\Gamma_p(3)}{\Gamma_p(5)}.
\end{align*}

For the coefficient of $|x|_p^3$ we have
\begin{align*}
& c_5\frac{\Gamma_p(6)}{\Gamma_p(4)}=Ec_3 \\&\quad
c_5=0.
\end{align*}

Generally, 
\begin{align}
& c_{2n+1}=0,
\\&\quad
c_{2n+4}\frac{\Gamma_p(2n+5)}{\Gamma_p(2n+3)}+Bc_{2n}=Ec_{2n+2}
\end{align}
for any $n\in\mathbb{N}$.

We now explore the case when $x\in\mathbb{Q}_p\setminus\mathbb{Z}_p$. We expect to find equations like (29), (30), (31). With these six equations together, we can prove the convergence of $\Psi_0$ later.

Let $x\in\mathbb{Q}_p\setminus{\mathbb{Z}_p}$. Then,

\begin{align*}
D^2\Psi_0(x)=\sum_{k=1}^{\infty} k_n \left(\frac{\Gamma_p(-n+1)}{\Gamma_p(-n-1)}|x|_p^{-n-2}-\frac{1}{\Gamma_p(-2)}\frac{p-1}{p-p^n}|x|_p^{-3}\right).
\end{align*}
We replace this expression into the equation and get
\begin{align*}
\sum_{n=1}^{\infty} & k_n \left(\frac{\Gamma_p(-n+1)}{\Gamma_p(-n-1)}|x|_p^{-n-2}-\frac{1}{\Gamma_p(-2)}\frac{p-1}{p-p^n}|x|_p^{-3}\right)
\\&\quad    
+\sum_{n=0}^{\infty}c_n\frac{1}{\Gamma_p(-2)}|x|_p^{-3}\frac{p-1}{p-p^{-n}}+B\sum_{n=1}^{\infty}k_n|x|_p^{-n+2}=E\sum_{n=1}^{\infty}k_n|x|_p^{-n}.
\end{align*}

For the coefficient of $|x|_p^{-3}$ we have
\begin{equation}
-\sum_{n=1}^{\infty}k_n\frac{1}{\Gamma_p(-2)}\frac{p-1}{p-p^n}+\sum_{n=0}^{\infty}c_n\frac{1}{\Gamma_p(-2)}\frac{p-1}{p-p^{-n}}=-Bk_5.
\end{equation}

Equations (29) and (32) are two equations that can help us determine convergence of $\Psi_0$ and the value of $E$ later. In addition, we need to find the general expression for $k_n$ like in equation (30) and (31).

For the coefficient of $|x|_p^1$ we have
\begin{align*}
 &Bk_1=0  \\&\quad  k_1=0.
\end{align*}

For the coefficient of $|x|_p^0$ we have

\begin{align*}
&Bk_2=0   \\&\quad   k_2=0.
\end{align*}

For the coefficient of $|x|_p^{-1}$ we have
\begin{align*}
&Bk_3=Ek_1  \\&\quad   k_3=0.
\end{align*}

For the coefficient of $|x|_p^{-2}$ we have
\begin{align*}
&Bk_4=Ek_2  \\&\quad  k_4=0. 
\end{align*}

For the coefficient of $|x|_p^{-4}$ we have
\begin{align*}
&k_2\frac{\Gamma_p(-1)}{\Gamma_p(-3)}+Bk_6=Ek_4
\\&\quad  k_6=0.
\end{align*}

For the coefficient of $|x|_p^{-5}$ we have
\begin{align*}
&k_3\frac{\Gamma_p(-2)}{\Gamma_p(-4)}+Bk_7=Ek_5
\\&\quad k_7=\frac{Ek_5}{B}.
\end{align*}

More generally,
\begin{align}
& k_{2n}=0,
\\
&k_{2n+1}\frac{\Gamma_p(-2n)}{\Gamma_p(-2n-2)}+Bk_{2n+5}=Ek_{2n+3}
\end{align}
for any positive integer $n$.

Given (30), (31), (33), (34), we know all coefficients $c_n$, $k_n$ in terms of $c_0$, $k_5$. Thus, we make the following definition.
\begin{mydef}
Define $\tau$ and $s$ as the sequences satisfying \begin{align}
&c_{2n}=c_0\tau_{2n},
\\&\quad
k_{2n+1}=k_5s_{2n+1} 
\end{align}
for $n\in\mathbb{N}\cup\{0\}$, where $c_n,k_n$ are the ones defined by (30),(31),(33),(34).
\end{mydef}

The following lemma, approximating the explicit expression of $\tau_{2n}$ and $s_{2n+1}$ for large primes $p$ and $B=1$, helps prove the convergence of the power series $\Psi_0$ (so that it becomes a solution to the $p$-adic Schrodinger Equation), the main goal of Example 4.
\begin{lem}
Given $lim_{n\rightarrow\infty}\frac{\Gamma_p(2n+3)}{\Gamma_p(2n+5)}=\frac{1}{p^2}$, assuming $p$ is large enough,
$B=1$, 
then 
\begin{equation}
\tau_{4n}=(-1)^n\frac{1}{p^{2n}}+O \left( \frac{1}{p^{2n+2}} \right) 
\end{equation}
for $n\geq1$, 
\begin{equation}
\tau_{4n+2}=(-1)^n\frac{nE}{p^{2n+2}}+O \left(\frac{1}{p^{2n+4}}\right)
\end{equation}
for $n\geq2$,
\begin{equation}
s_{4n+1}=(-1)^{n+1}p^{2n-2}+O(p^{2n-4})
\end{equation}
\begin{equation}
s_{4n+3}=(-1)^{n+1}nEp^{2n-2}+O(^{2n-4})
\end{equation}
 for $n\geq2$, where O represent the error terms.
\end{lem}

\begin{proof}

We claim that  $\frac{\Gamma_p(2n+3)}{\Gamma_p(2n+5)}\approx\frac{1}{p^2}$ for any $n\in\{0\}\cup\mathbb{N}$, $\frac{\Gamma_p(-2n)}{\Gamma_p(-2n-2)}\approx p^2$ for any $n\in\mathbb{N}$, given sufficiently large prime number $p$.

It is easy to see this claim is true, since all terms are negligible except the ones with the largest powers in the denominator and numerator as $p$ increases. Therefore, for a fixed $n$ we have

\begin{equation*}
\frac{\Gamma_p(2n+3)}{\Gamma_p(2n+5)}=\frac{1-p^{-2n-5}-p^{2n+2}+p^{-3}}{1-p^{2n+4}-p^{-2n-3}+p}\approx \frac{-p^{2n+2}}{-p^{2n+4}}=\frac{1}{p^2},
\end{equation*}

\begin{equation*}
\frac{\Gamma_p(-2n)}{\Gamma_p(-2n-2)}=\frac{1-p^{2n+1}-p^{-2n}+p}{1-p^{-2n-2}-p^{2n-1}+p^{-3}}\approx\frac{-p^{2n+1}}{-p^{2n-1}}=p^2.
\end{equation*}

Given this claim, for $B=1$, $c_{2n}=\tau_{2n}c_0$ and $k_{2n+1}=k_5s_{2n+1}$, we choose a sufficiently large prime number $p$ and rewrite equations (31) and (34) as the following:
\begin{equation}
\tau_{2n+4}p^2+\tau_{2n}=E\tau_{2n+2}
\end{equation}
\begin{equation}
s_{2n+1}p^2+s_{2n+5}=Es_{2n+3}.
\end{equation}
By come calculation using Lemma 2 above,

\begin{align*}
& \tau_4=\frac{1}{p^2}(-1),
\\&\quad
\tau_6=-E\frac{1}{p^4}, 
\\&\quad
\tau_8=\frac{1}{p^4}-E^2\frac{1}{p^6},
\\&\quad
\tau_{10}=2E\frac{1}{p^6}-E^3\frac{1}{p^8},
\\&\quad
\tau_{12}=\frac{E}{p^6}+O \left(\frac{1}{p^8}\right)
\\&\quad
\tau_{14}=-3E\frac{1}{p^8}+O \left(\frac{1}{p^{10}}\right)
\end{align*}

We find a pattern that matches $\tau_{4n}=(-1)^n\frac{1}{p^{2n}}+O \left( \frac{1}{p^{2n+2}} \right)$ for $n\geq1$, $\tau_{4n+2}=(-1)^n\frac{nE}{p^{2n+2}}+O(\frac{1}{p^{2n+4}})$ for $n\geq2$. We prove by induction that the pattern holds true for all $n\geq2, n\in\mathbb{N}$.

Assume
\begin{equation*}
\tau_{4n}=(-1)^n\frac{1}{p^{2n}}+O \left( \frac{1}{p^{2n+2}} \right), \tau_{4n+2}=(-1)^n\frac{nE}{p^{2n+2}}+O(\frac{1}{p^{2n+4}}).
\end{equation*}
We claim 
\begin{align*}
&\tau_{4n+4}=(-1)^{n+1}\frac{1}{p^{2n+2}}+O \left( \frac{1}{p^{2n+4}} \right), 
\\&\quad
\tau_{4n+6}=(-1)^{n+1}\frac{(n+1)E}{p^{2n+4}}+O \left(\frac{1}{p^{2n+6}}\right).
\end{align*}
Using equation (41), 
\begin{align*}
&\tau_{4n+4}p^2+\tau_{4n}=E\tau_{4n+2}
\\&\quad
\tau_{4n+4}=\frac{E((-1)^n\frac{nE}{p^{2n+2}}+O(\frac{1}{p^{2n+4}}))-((-1)^n\frac{1}{p^{2n}}+O \left( \frac{1}{p^{2n+2}} \right))}{p^2}
\\&\quad
=\frac{-(-1)^n\frac{1}{p^{2n}}+O(\frac{1}{p^{2n+2}})}{p^2}=(-1)^{n+1}\frac{1}{p^{2n+2}}+O \left( \frac{1}{p^{2n+4}} \right)
\end{align*}
Similarly,
\begin{align*}
&\tau_{4n+6}p^2+\tau_{4n+2}=E\tau_{4n+4}
\\&\quad
\tau_{4n+6}=\frac{E((-1)^{n+1}\frac{1}{p^{2n+2}}+O \left( \frac{1}{p^{2n+4}} \right))-(-1)^n\frac{nE}{p^{2n+2}}-O(\frac{1}{p^{2n+4}})}{p^2}
\\&\quad
=(-1)^{n+1}\frac{(n+1)E}{p^{2n+4}}+O \left(\frac{1}{p^{2n+6}}\right)
\end{align*}
The proof for $s_{4n+1}, s_{4n+3}$ uses the exactly same inductive method, and is thus omitted.
\end{proof}

Recall that our final goal is to prove $\begin{cases} \sum_{n=0}^{\infty} c_n |x|_p^n & \text{ if } x \in \mathbb{Z}_p \\ \sum_{n=1}^{\infty} k_n |x|_p^{-n} & \text{ otherwise} \end{cases}$ is convergent so that $\Psi_0(x)$ is a solution to the function D$^2$$\Psi(x)$+B$\left| x\right|_p^2$$\Psi(x)$=E$\Psi(x)$. As we find out in the following theorem, the statement is true.

\begin{thm}
\begin{align*}
\psi_0(x) &= \sum_{n=0}^{\infty} c_n f_n (x) + \sum_{n=1}^{\infty} k_n g_{-n}(x) \\
&= \begin{cases} \sum_{n=0}^{\infty} c_n |x|_p^n & \text{ if } x \in \mathbb{Z}_p \\ \sum_{n=1}^{\infty} k_n |x|_p^{-n} & \text{ otherwise}. \end{cases}
\end{align*}
is a solution to the function $D^2\Psi(x)+B\left| x\right|_p^2\Psi(x)=E\Psi(x)$, where $c_n,k_n$ are defined by (30), (31), (33), (34), $B=1$, E is some real constant determined by B.
\end{thm}

\begin{proof}
Fix $x\in\mathbb{Z}_p$ so that $|x|_p\leq1$. Then,

\begin{align*}
&\sum_{n=0}^{\infty} c_n |x|_p^n=c_0|x|_p^0+\sum_{n=1}^\infty \left((-1)^n\frac{1}{p^{2n}}+O\left(\frac{1}{p^{2n+2}}\right)\right)|x|_p^{4n}
\\&\quad
+c_6|x|_p^6+\sum_{n=2}^\infty \left((-1)^n\frac{nE}{p^{2n+2}}+O\left(\frac{1}{p^{2n+4}}\right)\right)|x|_p^{4n+2}.
\end{align*}

Since the error term O has smaller absolute value than the approximation, we have 
\begin{equation*}
\Big|\left((-1)^n\frac{1}{p^{2n}}+O\left(\frac{1}{p^{2n+2}}\right)\right)|x|_p^{4n}\Big|<\frac{2}{p^{2n}}|x|_p^{4n}. 
\end{equation*}

Since the series made by the latter term is convergent by the Ratio Test, 
\begin{equation*}
\sum_{n=1}^{\infty} \left((-1)^n\frac{1}{p^{2n}}+O\left(\frac{1}{p^{2n+2}}\right)\right)|x|_p^{4n} 
\end{equation*}
is absolutely convergent by the Comparison Test, and therefore convergent.

Similarly, 
\begin{equation*}
\sum_{n=2}^\infty \left((-1)^n\frac{nE}{p^{2n+2}}+O\left(\frac{1}{p^{2n+4}}\right)\right)|x|_p^{4n+2}
\end{equation*}
is convergent by comparing it with $\sum_{n=2}^\infty \frac{2nE}{p^{2n+2}}|x|_p^{4n+2}.$ We conclude that 
\begin{equation*}
\sum_{n=0}^{\infty} c_n |x|_p^n
\end{equation*}
is convergent for all $x\in \mathbb{Z}_p$.

Now, we claim 
\begin{equation*}
\sum_{n=1}^{\infty} k_n |x|_p^{-n}
\end{equation*}
is convergent for any $x\in\mathbb{Q}_p\setminus{\mathbb{Z}_p}$. Fix $x\in\mathbb{Q}_p\setminus{\mathbb{Z}_p}$ so that $|x|_p\geq p$. Then,

\begin{align*}
&\sum_{n=1}^{\infty} k_n |x|_p^{-n}=\sum_{n=1}^{6} k_n |x|_p^{-n}+\sum_{n=2}^\infty ((-1)^{n+1}p^{2n+2}+O(p^{2n-4}))|x|_p^{-4n-1}
\\&\quad
+\sum_{n=2}^\infty ((-1)^{n+1}nEp^{2n+2}+O(p^{2n-4}))|x|_p^{-4n-3}.
\end{align*}

Again, using the condition that the error term O has smaller absolute value than the main term and $|x|_p\geq p$, we have 
\begin{equation*}
((-1)^{n+1}p^{2n+2}+O(p^{2n-4}))|x|_p^{-4n-1}<\frac{2p^{2n+2}}{p^{4n+1}}=2p^{-2n+1}.
\end{equation*}
 Obviously, $\sum_{n=2}^\infty 2p^{-2n+1}$ converges by the Geometric Series Test.

Thus, 
\begin{equation*}
\sum_{n=2}^\infty ((-1)^{n+1}p^{2n+2}+O(p^{2n-4}))|x|_p^{-4n-1}\end{equation*} 
is convergent by the Comparison Test.

Similarly, 
\begin{equation*}
\sum_{n=2}^\infty ((-1)^{n+1}nEp^{2n+2}+O(p^{2n-4}))|x|_p^{-4n-3}
\end{equation*}
is convergent by comparing it with 
\begin{equation*}
\sum_{n=2}^\infty 2nEp^{2n+2}p^{-4n-3}=\sum_{n=2}^\infty 2nEp^{-2n-1}, 
\end{equation*}
of which the ratio $|\frac{n+1}{n}p^{-2}|$ converges to $p^{-2}<1$ as $n\rightarrow\infty$.

Thus, 
\begin{equation*}
\sum_{n=1}^\infty k_n|x|_p^{-n} 
\end{equation*}
converges.
\end{proof}

In the case $B=1$, the above is enough to prove $\Psi_0(x)$ is a solution to our equation, we still wish to approximate the value of $E$ asymptotically for large prime numbers $p$. Before that, we derive an important equation that would help us estimate $E$ based on $B$.

 We rewrite the condition on $E$ as two equations using $\tau_{n}$ and $s_{n}$:
\begin{align}
& c_0(\sum_{n=0}^{\infty}\frac{\tau_{2n}}{p-p^{-2n+3}}-\frac{E\Gamma_p(-2)}{p-1})=k_5\sum_{n=2}^\infty\frac{s_{2n+1}}{p-p^{2n+4}}  \\&\quad
c_0\sum_{n=0}^\infty\frac{\tau_{2n}}{p-p^{-2n}}=k_5(\sum_{n=2}^\infty\frac{s_{2n+1}}{p-p^{2n+1}}-B\frac{\Gamma_p(-2)}{p-1}).
\end{align}

Denote 
\begin{align*}
A &=\sum_{n=0}^\infty\frac{\tau_{2n}}{p-p^{-2n}} \\
F &=\sum_{n=2}^\infty\frac{s_{2n+1}}{p-p^{2n+1}}-B\frac{\Gamma_p(-2)}{p-1} \\
C &=\sum_{n=0}^\infty\frac{\tau_{2n}}{p-p^{-2n+3}}-\frac{E\Gamma_p(-2)}{p-1} \\
D &=\sum_{n=2}^\infty\frac{s_{2n+1}}{p-p^{2n+4}}.
\end{align*}
Then, we have two linear equations in $c_0$ and $k_5$, 
\begin{align*}
& c_0A-k_5F=0   \\&\quad  c_0C-k_5D=0
\end{align*}

We do not wish to have trivial solutions to the system of equations, since it is meaningless for exploring the $p$-adic analog of the solution. Thus, in order to let the two equations have a non-zero solution, we need $det
\begin{pmatrix}
A&-F\\C&-D
\end{pmatrix}=0$. Thus, $$AD=FC.$$
Consequently, we find
\begin{equation}
\begin{split}
&\left(\sum_{n=0}^\infty \frac{\tau_{2n}} {p-p^{-2n}} \right) \left(\sum_{n=2}^\infty\frac{s_{2n+1}}{p-p^{2n+4}} \right)= 
\\
&( \sum_{n=2}^\infty\frac{s_{2n+1}}{p-p^{2n+1}}-B\frac{\Gamma_p(-2)}{p-1})(\sum_{n=0}^\infty\frac{\tau_{2n}}{p-p^{-2n+3}}-\frac{E\Gamma_p(-2)}{p-1})
\end{split}
\end{equation}

\begin{cor}
Given $B=1$ and the conditions of Theorem 18, $$E=2-\frac{2}{p}+O \left(\frac{1}{p^2}\right)$$ gives one solution $\Psi_0 (x)$ as above, but this solution is not necessarily unique.
\end{cor}

\begin{proof}
Assume $E=C+a(\frac{1}{p})+O \left(\frac{1}{p^2}\right)$ for some constant $C, a$. Here $O$ represents the size of the error term.

By Equation (37) and (38), 
\begin{equation}
\begin{split}
&\sum_{n=0}^\infty \frac{\tau_{2n}}{p-p^{-2n}}=\frac{\tau_0}{p-p^0}+\frac{\tau_{2}}{p-p^{-2}}+\frac{\tau_4}{p-p^{-4}}+\frac{\tau_6}{p-p^{-6}}+...
\\
&=\frac{1}{p-1}+0+\frac{(-1)p^2}{p-p^{-4}}+\frac{(-1)E/p^4}{p-p^{-6}}+...\approx\frac{1}{p}-\frac{1}{p^3}+O \left(\frac{1}{p^5}\right).
\end{split}
\end{equation}
By Equation (39) and (40),
\begin{equation}
\begin{split}
&\sum_{n=2}^\infty\frac{s_{2n+1}}{p-p^{2n+4}}=\frac{s_5}{p-p^8}+\frac{s_7}{p-p^{10}}+\frac{s_9}{p-p^{12}}+...
\\&\quad
\approx\frac{1}{p-p^8}+\frac{E}{p-p^{10}}+\frac{(-1)p^2}{p-p^{12}}+...
\approx\frac{-1}{p^8}+\frac{-E+1}{p^{10}}+O\left(\frac{1}{p^{12}}\right).
\end{split}
\end{equation}
Replacing (46), (47) into the left hand side of Equation (45),
\begin{equation}
\begin{split}
&LHS=\left(\frac{1}{p}-\frac{1}{p^3}+O\left(\frac{1}{p^5}\right)\right)\left(\frac{-1}{p^8}+\frac{-E+1}{p^{10}}+O\left(\frac{1}{p^{12}}\right)\right)
\\
&\approx \frac{-1}{p^9}+\frac{-E+2}{p^{11}}+O\left(\frac{1}{p^{13}}\right).
\end{split}
\end{equation} 
We have
\begin{equation}
\begin{split}
&\sum_{n=2}^\infty \frac{s_{2n+1}}{p-p^{2n+1}}-\frac{\Gamma_p(-2)}{p-1}=\frac{s_5}{p-p^5}+\frac{s_7}{p-p^7}+\frac{s_9}{p-p^9}-\frac{1-p^{-3}}{(1-p^2)(p-1)}
\\&
\approx \frac{1}{p-p^5}+\frac{E}{p-p^7}+\frac{(-1)p^2}{p-p^9}+\frac{1}{p^3}+\frac{1}{p^4}+\frac{2}{p^5}+...
\\&
\approx\frac{1}{p^3}+\frac{1}{p^4}+\frac{1}{p^5}+O\left(\frac{1}{p^6}\right).
\end{split}
\end{equation}
Similarily,
\begin{equation}
\begin{split}
&\sum_{n=0}^\infty \frac{\tau_{2n}}{p-p^{-2n+3}}-\frac{E\Gamma_p(-2)}{p-1}
\\&\quad
=\frac{1}{p-p^3}+\frac{\frac{-1}{p^2}}{p-p^{-1}}+\frac{-E/p^4}{p-p^{-3}}+\frac{E}{p^3}+\frac{E}{p^4}+...
\\&\quad
\approx \frac{-2+E}{p^3}+\frac{E}{p^4}+O\left(\frac{1}{p^5}\right).
\end{split}
\end{equation}

Replacing (49) and (50) into the right hand side of Equation (45),
\begin{equation}
\begin{split}
&RHS=(\frac{1}{p^3}+\frac{1}{p^4}+\frac{1}{p^5}+O(\frac{1}{p^6}))(\frac{-2+E}{p^3}+\frac{E}{p^4}+O(\frac{1}{p^5}))
\\
&\approx \frac{E-2}{p^6}+\frac{2E-2}{p^7}+\frac{3E-2}{p^8}+\frac{4E-2}{p^9}...
\end{split}
\end{equation}

Since $LHS=RHS$,we replace $E=c+a(\frac{1}{p})+O(\frac{1}{p^2})$ into the equation and have
\begin{align*}
&\frac{-1}{p^9}+O(\frac{1}{p^{10}})=\frac{E-2}{p^6}+\frac{2E-2}{p^7}+\frac{3E-2}{p^8}+\frac{4E-2}{p^9}+...
\\&\quad
=\frac{c-2}{p^6}+\frac{a+2c-2}{p^7}+...
\end{align*}

Since $c-2=0$, $a+2c-2=0$, we have $c=2$, $a=-2$.
Thus, $$E=2-\frac{2}{p}+O(\frac{1}{p^2}).$$
\end{proof}
\end{ex}

\begin{ex}
Inspired by Example 4, we can find another easy solution to 
\begin{equation*}
D^2\Psi(x)+B\left| x\right|_p^2\Psi(x)=E\Psi(x)
\end{equation*}
 with $B=-1$, \begin{align*}
\psi_0(x) &= \sum_{n=0}^{\infty} c_n f_n (x) + \sum_{n=1}^{\infty} k_n g_{-n}(x) \\
&= \begin{cases} \sum_{n=0}^{\infty} c_n |x|_p^n & \text{ if } x \in \mathbb{Z}_p \\ \sum_{n=1}^{\infty} k_n |x|_p^{-n} & \text{ otherwise}. \end{cases}
\end{align*}.

\begin{lem}
Given $lim_{n\shortrightarrow\infty}\frac{\Gamma_p(2n+3)}{\Gamma_p(2n+5)}=\frac{1}{p^2}$, the equation as well as $\psi_0$ in Example 5, assuming $p$ is large enough, $B=-1$, $\tau_n,s_n$ be the same as Definition 20, then \begin{equation}
\tau_{4n}=\frac{1}{p^{2n}}+O \left( \frac{1}{p^{2n+2}} \right) 
\end{equation}
for $n\geq1$, 
\begin{equation}
\tau_{4n+2}=\frac{nE}{p^{2n+2}}+O \left(\frac{1}{p^{2n+4}}\right)
\end{equation}
for $n\geq2$,
\begin{equation}
s_{4n+1}=p^{2n-2}+O(p^{2n-4})
\end{equation}
\begin{equation}
s_{4n+3}=nEp^{2n-2}+O(p^{2n-4})
\end{equation}
 for $n\geq2$, where O are error terms.
\end{lem}

\begin{proof}
The proof method is exactly the same as Lemma 2, using the inductive method on the recursive relationship, i.e. Equation (30),(31),(33),(34) of $c_n$ and $k_n$. It is left to readers to find out Lemma 3 is true.
\end{proof}

\begin{thm}
\begin{align*}
\psi_0(x) &= \sum_{n=0}^{\infty} c_n f_n (x) + \sum_{n=1}^{\infty} k_n g_{-n}(x) \\
&= \begin{cases} \sum_{n=0}^{\infty} c_n |x|_p^n & \text{ if } x \in \mathbb{Z}_p \\ \sum_{n=1}^{\infty} k_n |x|_p^{-n} & \text{otherwise}. \end{cases}
\end{align*}
is a solution to the function $D^2\Psi(x)+B\left| x\right|_p^2\Psi(x)=E\Psi(x)$, where $B=-1$, $E$ is some constant determined by $B$.
\end{thm}

\begin{proof}
Again, the proof method is exactly the same as Theorem 18, using the Comparison Test, the Geometric Series Test and the Ratio Test. 
For each $n\in\mathbb{N}$, \begin{equation*}
\tau_{4n}=\frac{1}{p^{2n}}+O \left( \frac{1}{p^{2n+2}} \right)<\frac{2}{p^{2n}}.
\end{equation*}
\begin{equation*}
\tau_{4n+2}=\frac{nE}{p^{2n+2}}+O(\frac{1}{p^{2n+4}})<\frac{2nE}{p^{2n+2}}.
\end{equation*}
Since $\sum\frac{2}{p^{2n}}$ converges by the Geometric Series Test, $\sum c_{4n}|x|_p^{4n}$ converges.
Since $\sum \frac{2nE}{p^{2n+2}}$ converges by the Ratio Test, $\sum c_{4n+2}|x|_p^{4n+2}$ converges. Thus, $\sum c_n|x|_p^n$ converges. Similar methods can prove $k_n|x|_p^n$ converges and the proof is thus omitted.
\end{proof}

Like Example 4, let us find a constant E that satisfies the $p$-adic Schr\"{o}dinger equation under the condition $B=-1$.
\begin{cor}
Given the condition in Theorem 19, $E=-\frac{2}{3}\frac{1}{p^2}+\frac{7}{3}\frac{1}{p^3}+O(\frac{1}{p^4})$ is one solution to $B=-1$, but not necessarily unique.
\end{cor}

\begin{proof}
We omit some procedures of calculation since they are very similar to those under Corollary 2. Assume $E=C+\frac{a}{p}+b(\frac{1}{p^2})+d(\frac{1}{p^3})+O(\frac{1}{p^4})$ for some constant $C, a, b, d$.
\begin{equation}
\sum_{n=0}^\infty \frac{\tau_{2n}}{p-p^{-2n}}\approx \frac{1}{p}+\frac{1}{p^3}+O\left(\frac{1}{p^5}\right)
\end{equation}

\begin{equation}
\sum_{n=2}^\infty \frac{s_{2n+1}}{p-p^{2n+4}}\approx \frac{-1}{p^8}+\frac{-E-1}{p^{10}}+O\left(\frac{1}{p^{12}}\right)
\end{equation}

Replacing (56) and (57) into the left hand side of (45), \begin{equation}
LHS\approx \frac{-1}{p^9}+\frac{-E-2}{p^{11}}+O\left(\frac{1}{p^{13}}\right)
\end{equation}

\begin{equation}
\sum_{n=2}^\infty \frac{s_{2n+1}}{p-p^{2n+1}}+\frac{\Gamma_p(-2)}{p-1}\approx \frac{-1}{p^3}+\frac{-1}{p^4}+\frac{-3}{p^5}+O\left(\frac{1}{p^6}\right)
\end{equation}

\begin{equation}
\sum_{n=0}^\infty \frac{\tau_{2n}}{p-p^{-2n+3}}-\frac{E\Gamma_p(-2)}{p-1}\approx \frac{E}{p^3}+\frac{E}{p^4}+\frac{3E}{p^5}+O\left(\frac{1}{p^6}\right).
\end{equation}

Replacing (59) and (60) into the right hand side of (45), \begin{equation}
RHS\approx\frac{-E}{p^6}+\frac{-2E}{p^7}+\frac{-7E}{p^8}+O\left(\frac{1}{p^9}\right).
\end{equation}

Since $LHS=RHS$, we replace $E=C+\frac{a}{p}+b(\frac{1}{p^2})+d(\frac{1}{p^3})+O(\frac{1}{p^4})$ into the equation, and get $$E=-\frac{2}{3}\frac{1}{p^2}+\frac{7}{3}\frac{1}{p^3}+O \left(\frac{1}{p^4}\right)$$ 
after some calculation.
\end{proof}

\end{ex}

\begin{remark}
After exploring the above two examples assuming $\Psi_0$ is a convergent power series, it is intuitive to think if constant multiples of $\Psi_0$ are solutions to the $p$-adic Schrodinger Equation. Thus, we have the following proposition.
\end{remark}

\begin{prop}
If \begin{align}
\psi_0(x) &= \sum_{n=0}^{\infty} c_n f_n (x) + \sum_{n=1}^{\infty} k_n g_{-n}(x) \\
&= \begin{cases} \sum_{n=0}^{\infty} c_n |x|_p^n & \text{ if } x \in \mathbb{Z}_p \\ \sum_{n=1}^{\infty} k_n |x|_p^{-n} & \text{ otherwise}. \end{cases}
\end{align} is a solution to 
\begin{equation*}
D^2\Psi(x)+B\left| x\right|_p^2\Psi(x)=E\Psi(x),
\end{equation*} with $B,E$ fixed, then $\Psi_c(x)=c\Psi_0(x)$ is also a solution to the equation for any $c\in\mathbb{Q}_p$.
\end{prop}

\begin{proof}
Assume $x\in\mathbb{Z}_p$, then we replace $\Psi_0(x)=\sum_{n=0}^\infty c_n|x|_p^n$ into the equation, and get
\begin{equation}
\begin{split}
&\sum_{n=0}^{\infty}c_n(\frac{\Gamma_p(n+1)}{\Gamma_p(n-1)}|x|_p^{n-2}+\frac{1}{\Gamma_p(-2)}\frac{p-1}{p-p^{-n+3}})+\sum_{n=1}^{\infty}k_n\frac{1}{\Gamma_p(-2)}\frac{1-p}{p-p^{n+3}}
\\
&+B|x|_p^2\sum_{n=0}^\infty c_n|x|_p^n=E\sum_{n=0}^\infty c_n|x|_p^n.
\end{split}
\end{equation}

We multiply both sides of Equation (64) above by a constant $c\in\mathbb{Q}_p$, and have
\begin{equation}
\begin{split}
&c\sum_{n=0}^{\infty}c_n(\frac{\Gamma_p(n+1)}{\Gamma_p(n-1)}|x|_p^{n-2}+\frac{1}{\Gamma_p(-2)}\frac{p-1}{p-p^{-n+3}})+c\sum_{n=1}^{\infty}k_n\frac{1}{\Gamma_p(-2)}\frac{1-p}{p-p^{n+3}}
\\
&+cB|x|_p^2\sum_{n=0}^\infty c_n|x|_p^n=cE\sum_{n=0}^\infty c_n|x|_p^n.
\end{split}
\end{equation}

It is easy to check that the left hand side of $Equation(65)=D^2(\Psi_c(x))+B|x|_p^2\Psi_c(x)$, and the right hand side of $ Equation(65)=E\Psi_c(x)$. 

Thus, 
\begin{equation}
D^2(\Psi_c(x))+B|x|_p^2\Psi_c(x)=E\Psi_c(x)
\end{equation}
The case when $x\in\mathbb{Q}_p\setminus\mathbb{Z}_p$ is similar. We can also get Equation (66) in this case. Thus, $\Psi_c$ is indeed a solution to the equation.
\end{proof}

\pagebreak

\section*{Acknowledgments}

This undergraduate thesis was finished by the author under the supervision of Dr. Maria Nastasescu at Northwestern University Math Department during the Summer of 2022. The author appreciates her help and support.

\vspace{1cm}

\end{document}